%% file: mercthesis.tex
\documentclass[edeposit,fullpage,offcenter]{uiucecethesis08}
\usepackage{amsmath,amssymb,amsthm}
\usepackage{graphicx}
\newcommand{\surfint}{{\int_{\partial\Omega}}}
\newcommand{\BB}{{\mathbb{B}^d}}
\newcommand{\RR}{{\mathbb{R}}}
\newcommand{\DD}{{\mathcal{D}}}
\newcommand{\TT}{{\mathrm T}}
\renewcommand{\epsilon}{{\varepsilon}}

\newcommand{\rhat}{{\hat{r}}}
\newcommand{\nhat}{{\hat{n}}}
\newcommand{\that}{{\hat{t}}}
\newcommand{\thetahat}{{\hat{\theta}}}
\newcommand{\rp}{{\rho^\prime}}
\newcommand{\rpp}{{\rho^{\prime\prime}}}
\newcommand{\Rp}{{R^{\prime}}}
\newcommand{\Rpp}{{R^{\prime\prime}}}

\newcommand{\ainf}{{a_\infty}}
\newcommand{\ainfsq}{{a_\infty^2}}

\newcommand{\ostar}{{\omega^*}}
\newcommand{\Ostar}{{\Omega^*}}
\newcommand{\Ls}{{\Delta_S}}
\newcommand{\Ns}{{\nabla_{\!\!S}}}
\newcommand{\Yl}{{Y_l}}

\newcommand{\surfintB}{{\int_{\partial\BB}}}

\newcommand{\ubar}{{\overline{u}}}

\DeclareMathOperator{\BigO}{O}
\DeclareMathOperator{\trace}{tr}
\DeclareMathOperator{\grad}{grad}
\DeclareMathOperator{\mydiv}{div}
\DeclareMathOperator{\sign}{sign}

\newcommand{\sproj}{{P_{\partial\Omega}}}
\newcommand{\sgrad}{{\grad_{\partial\Omega}}}
\newcommand{\sdiv}{{\mydiv_{\partial\Omega}}}
\newtheorem{thm}{Theorem}

\newtheorem{lemma}{Lemma}[chapter]
\newtheorem{prop}{Proposition}
\newtheorem{fact}[lemma]{Fact}
\newtheorem{cor}[lemma]{Corollary}

\theoremstyle{remark}
\newtheorem*{rmk}{Remark}

\newtheoremstyle{newdiffnotenonum}{}{}{\itshape}{}{\bfseries}{.}{ }%
  {\thmname{#1}\thmnote{( \mdseries #3)}}
\theoremstyle{newdiffnotenonum}

\newtheorem{thm2pp}{Theorem 2$\,^{\prime\prime}$}

\begin{document}
\title{An isoperimetric inequality for fundamental tones of free plates}
\author{Laura Chasman}
\department{Mathematics}
\schools{B.S., California Institute of Technology, 2003}
\phdthesis
\advisor{Richard Laugesen}
\degreeyear{2009}
\committee{Associate Professor Jared Bronski, Chair \\ Associate Professor Dirk Hundertmark\\ Associate Professor Richard Laugesen\\ Assistant Professor Eduard Kirr}
\maketitle 
\frontmatter
\begin{abstract}
We establish an isoperimetric inequality for the fundamental tone (first nonzero eigenvalue) of the free plate of a given area, proving the ball is maximal. Given $\tau>0$, the free plate eigenvalues $\omega$ and eigenfunctions $u$ are determined by the equation $\Delta\Delta u-\tau\Delta u = \omega u$ together with certain natural boundary conditions. The boundary conditions are complicated but arise naturally from the plate Rayleigh quotient, which contains a Hessian squared term $|D^2u|^2$.

We adapt Weinberger's method from the corresponding free membrane problem, taking the fundamental modes of the unit ball as trial functions. These solutions are a linear combination of Bessel and modified Bessel functions.
\end{abstract}
\chapter*{ACKNOWLEDGMENTS} I am grateful to the University of Illinois
Department of Mathematics and the Research Board for support during my
graduate studies, and the National Science Foundation for graduate student support
under grants DMS-0140481 (Laugesen) and DMS-0803120 (Hundertmark) and DMS 99-83160
(VIGRE), and the University of Illinois Department of Mathematics for travel support to attend the 2007 Sectional meeting of the AMS in New York. I would also like to thank the Mathematisches Forschungsinstitut Oberwolfach for travel support to attend the workshop on Low Eigenvalues of Laplace and Schr\"{o}dinger Operators in 2009.

\tableofcontents

\mainmatter
\include{intro}

\include{spectrum}
\include{fundmodtens}

\include{bessel}

\include{unitball}
\include{maintheorem}

\include{1dim}
\include{futuredirec}
\appendix*
\include{calcfact}
\backmatter

\end{document}

%% file: intro.tex
\chapter{Plates and the isoperimetric problem}
Isoperimetric problems are about minimizing or maximizing a quantity subject to constraints. The classical isoperimetric inequality states that of all planar regions of the same perimeter, the disk has maximal area. Equivalently, of all regions of the same area, the disk minimizes perimeter. The three-dimensional version can be observed physically quite readily - many mammals sleep curled up in a ball, while keeping their volume the same, to minimize surface area and hence heat loss.

Many physical quantities satisfy isoperimetric-type inequalities. The goal of this thesis is to prove an isoperimetric result for a free plate under tension with unconstrained edges: of all such plates having the same area, the disk has the highest fundamental pitch.

Researchers have investigated and proved isoperimetric inequalities regarding frequencies of vibration in related situations. Lord Rayleigh conjectured, and Faber and Krahn proved, that of all membranes of the same area with constrained edges, a circular drum produces the lowest pitch. Kornhauser and Stakgold conjectured the opposite bound for a membrane with unconstrained edges; this result was proven by Szeg\H o and Weinberger. This thesis generalizes their result to plates under tension. Plate problems are more difficult than membrane problems because they involve the bi-Laplacian rather than the Laplacian.

\subsection*{Mathematical formulation}
We now develop the mathematical formulation of the free plate isoperimetric problem.
Let $\Omega$ be a smoothly bounded region in $\RR^d$, $d\geq 2$, and fix a parameter $\tau > 0$. The ``plate'' Rayleigh quotient is
\begin{equation}
Q[u] = \frac{\int_\Omega |D^2 u|^2 + \tau |D u|^2\,dx}{\int_\Omega |u|^2\,dx}. \label{RQ}
\end{equation}
Here $|D^2u|=(\sum_{jk}u_{x_jx_k}^2)^{1/2}$ is the Hilbert-Schmidt norm of the Hessian matrix $D^2u$ of $u$, and $Du$ denotes the gradient vector.

Physically, when $d=2$ the region $\Omega$ is the shape of a homogeneous, isotropic plate. The parameter $\tau$ represents the ratio of lateral tension to flexural rigidity of the plate; for brevity we refer to $\tau$ as the tension parameter. Positive $\tau$ corresponds to a plate under tension, while taking $\tau$ negative would give us a plate under compression. The function $u$ describes a transverse vibrational mode of the plate, and the Rayleigh quotient $Q[u]$ gives the bending energy of the plate.

From the Rayleigh quotient \eqref{RQ}, we will derive in Chapter~\ref{spectrum} the partial differential equation and boundary conditions governing the vibrational modes of a free plate. The critical points of \eqref{RQ} are the eigenstates for the plate satisfying the free boundary conditions and the critical values are the corresponding eigenvalues. The equation is:
\begin{equation}
\Delta \Delta u - \tau \Delta u = \omega u, \label{maineq}
\end{equation}
where $\omega$ is the eigenvalue, with the natural (\emph{i.e.}, unconstrained or ``free'') boundary conditions on $\partial\Omega$:
\begin{align}
&Mu := \frac{\partial^2 u}{\partial n^2}= 0\\
&Vu := \tau\frac{\partial u}{\partial n}-\sdiv\left(\sproj\left[(D^2u)n\right]\right)-\frac{\partial(\Delta u)}{\partial n} = 0
\end{align}
Here $n$ is the outward unit normal to the boundary and $\sdiv$ and $\sgrad$ are the surface divergence and gradient.

The eigenvalue equation \eqref{maineq} can also be obtained by separating the plate wave equation
\[
\phi_{tt}=-\Delta\Delta\phi+\tau\Delta\phi,
\]
by the separation $\phi(x,t)=u(x)\cos(\sqrt{\omega}t)$. The eigenvalue $\omega$ is therefore the square of the frequency of vibration of the plate. The quantities appearing as boundary conditions have physical significance as well. The expression $Mu$ is the \emph{bending moment}. As the plate bends, one side compresses while the other expands, leading to a restoring moment which must vanish at an unconstrained edge. 

\subsection*{The problem}
We will prove in Chapter~\ref{spectrum} that the spectrum of the free plate under tension is discrete, consisting entirely of eigenvalues with finite multiplicity:
\[
 0=\omega_0<\omega_1\leq\omega_2\leq\dots\rightarrow\infty.
\]
 We also have a complete $L^2$-orthonormal set of eigenfunctions $u_0\equiv$ const, $u_1$, $u_2$, and so forth.

We call $u_1$ the \emph{fundamental mode} and the eigenvalue $\omega_1$ the \emph{fundamental tone}; the latter can be expressed using the Rayleigh-Ritz variational formula:
\[
\omega_1(\Omega)=\min\{Q[u]:u\in H^2(\Omega), \int_\Omega u\,dx=0\}.
\]
In general, the $k$th eigenvalue is the minimum of $Q[u]$ over the space of all functions $u$ $L^2$-orthogonal to the eigenfunctions $u_0$, $u_1$,$\dots$, $u_{k-1}$. Because $u_0$ is the constant function, the condition $u\perp u_0$ can be written $\int_\Omega u\,dx=0$.

Let $\Omega^*$ denote the ball with the same volume as $\Omega$. The main goal of this thesis is to prove the following theorem.
\begin{thm}\label{thm1} For all smoothly bounded regions of a fixed volume, the fundamental tone of the free plate with a given positive tension is maximal for a ball. That is, if $\tau>0$ then
\begin{equation}
\omega_1(\Omega) \leq \omega_1(\Omega^*), \qquad\text{with equality if and only if $\Omega$ is a ball.} \label{isoineq}
\end{equation}
\end{thm}

\begin{figure}
\begin{center}
\includegraphics{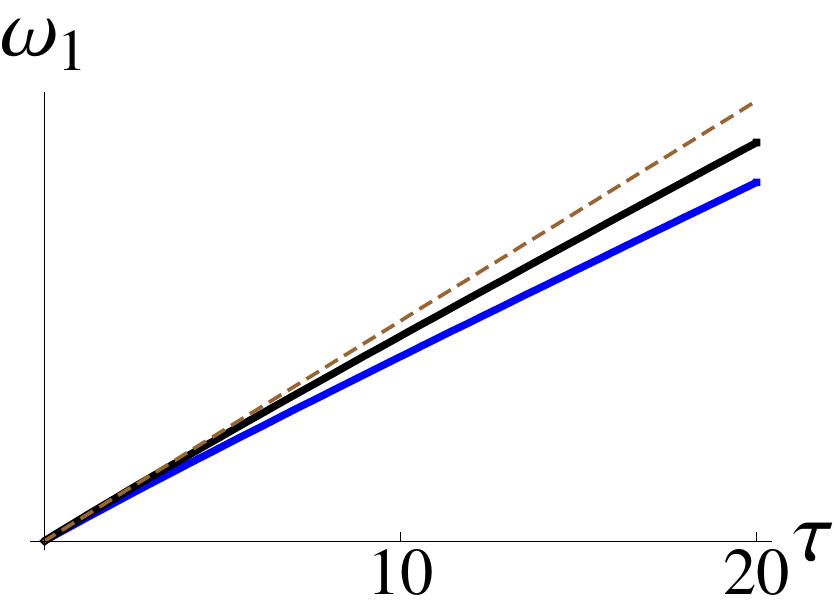}
\caption{The fundamental tone $\omega_1(\Ostar)$ of the disk (middle curve) and of another region $\omega_1(\Omega)$ (bottom curve). The dashed straight line is tangent to $\omega_1(\Ostar)$ at $\tau=0$.}
\end{center}
\end{figure}

In the limiting case $\tau=0$, the first $d+1$ eigenvalues of $\Omega$ are trivial because $Q[u]=0$ for all linear functions $u$. Thus we need the tension parameter $\tau$ to be positive to get a nontrivial conjecture.

The remainder of this chapter consists of a summary of the dissertation and then a brief history of related problems.

We examine the behavior of the spectrum in Chapter~\ref{spectrum}. In particular, we prove the spectrum is comprised only of eigenvalues of finite multiplicity, with an associated complete eigenbasis. We obtain some bounds on the fundamental tone as a function of tension $\tau$ in Chapter~\ref{fundmodtens}, where we also examine the fundamental tone in the extreme cases of infinite tension and zero tension. In addition, we derive the natural boundary conditions from the Rayleigh quotient.

A discussion of ultraspherical Bessel functions appears in Chapter~\ref{bessel}, along with a collection of recurrence relations and facts about them that we will need to prove our main theorem. We put these facts to immediate use in Chapter~\ref{unitball}, where we find the general form of eigenfunctions for the ball and establish the angular dependence of its fundamental mode, in Theorem~\ref{thm2}. Chapter~\ref{mainthm} presents the proof of our main result, Theorem~\ref{thm1}.

In Chapter~\ref{1dim}, we discuss the one-dimensional analogue of the free plate problem. Although there is no isoperimetric inequality in one dimension, we can prove an analogue of Theorem~\ref{thm2}, that the fundamental mode of the free rod is an odd function about its midpoint. Also discussed are the eigenfunctions of the free rod under compression ($\tau<0$). Building on that approach, In Chapter~\ref{furtherdirec}, possible future directions and generalizations of the main problem are discussed. Finally, in the Appendix, we gather some calculus facts used in prior chapters.

\subsection*{Brief history of isoperimetric problems}
Isoperimetric problems for eigenvalues of the Laplacian have fascinated researchers for quite some time \cite{Asummary,AB92,AL96,H06,Kawohlbook,Kesavan,Yau93}. In 1877, Lord Rayleigh \cite{RToS} conjectured an isoperimetric inequality for the first eigenvalue of a fixed membrane:
\[
\lambda_1(\Omega) \geq \lambda_1(\Ostar) \qquad\text{with equality if and only if $\Omega$ is a ball.}
\]
Here $\lambda_j$ is the $j$th eigenvalue satisfying the membrane equation $-\Delta u=\lambda u$ together with the boundary condition $u=0$ on $\partial\Omega$.
The above inequality was later proven by Faber \cite{Faber53} and Krahn \cite{krahn25,krahn26} and now bears their names. Kornhauser and Stakgold \cite{KS52} conjectured in 1952 the opposite result for the free membrane problem, $-\Delta u = \mu u$ with $\partial u/\partial n = 0$ on $\partial \Omega$. In this case, the eigenvalues are $0=\mu_0\leq \mu_1\leq\mu_2\dots$. The lowest eigenvalue $\mu_0=0$ corresponds to a constant eigenfunction. This mode does not vibrate; thus we call the next eigenvalue, $\mu_1$, the fundamental tone. The Kornhauser-Stakgold conjecture sought to maximize the fundamental tone of the free membrane:
\[
\mu_1(\Omega) \leq \mu_1(\Ostar)\qquad\text{with equality if and only if $\Omega$ is a ball.}
\]
This was proven by Szeg\H o \cite{S50,Serrata} and Weinberger \cite{W56}. Szeg\H o's proof uses conformal mapping and is only valid in two dimensions for simply connected regions. Weinberger's approach works for arbitrary domains in all dimensions. Our proof for Theorem~\ref{thm1} is inspired by Weinberger's approach, namely using trial functions and demonstrating monotonicity properties of the resulting quotient of integrals.

The bi-Laplacian operator appears in plate problems just as the Laplacian underpins membrane problems. The theory of the bi-Laplacian is not nearly so well developed as that of the Laplacian. For example, solvability of the biharmonic equation in Lipschitz domains with Neumann boundary conditions was established only a few years ago \cite{verchota}. Furthermore, the maximum principle fails for the bi-Laplacian, even in one dimension. Even so, isoperimetric problems for plate eigenvalues have long been under consideration. Lord Rayleigh \cite{RToS} conjectured an isoperimetric inequality for the clamped plate with zero tension, $\Delta^2u = \alpha u$, where ``clamped" refers to the boundary conditions $u=\partial u/\partial n = 0$ on $\partial \Omega$. This isoperimetric inequality (like that of the fixed membrane) gives a lower bound on the first eigenvalue:
\[
 \alpha_1(\Omega)\geq\alpha_1(\Ostar).
\]
 Work towards this result was begun by Szeg\H o \cite{S50} and continued by Talenti \cite{T81} (see also Mohr \cite{M75}). Nadirashvili \cite{N92,N95} proved the conjecture in two dimensions based on Talenti's work on rearrangement of elliptic partial differential equations. Ashbaugh and Benguria later extended the proof to three dimensions \cite{AB95}. The problem remains open for dimensions four and higher, with a partial result by Ashbaugh and Laugesen \cite{AL96}, and is open in all dimensions for $\tau>0$. For an overview of work on the clamped plate problem, see \cite[Chapter 11, p.\ 169--174]{H06} and \cite[p.\ 105--116]{Kesavan}

There are other boundary conditions for the plate besides the natural and clamped conditions discussed so far. The simply supported plate is governed by $\Delta^2-\tau\Delta u=\beta u$ with the requirement that $u=0$ on $\partial\Omega$; from here $Mu=0$ arises as a natural boundary condition. It is natural to conjecture an isoperimetric inequality for this problem too, but no work seems to have been done on it.

There is also a body of work on plate problems that does not focus on isoperimetric inequalities. General vibrating plate eigenvalue problems are discussed with experimental data by Leissa in \cite{L69}, including approximate solutions for the free rectangular plate. The buckling eigenvalues of a clamped plate have been considered by Payne \cite{P55,P61}. Supported plate work includes Payne \cite{P58} and Licar and Warner \cite{LW73}, who examine domain dependence of plate eigenvalues. The free plate without tension is considered in the same papers. Free plate work without tension also includes Nakata and Fujita, who establish upper and lower bounds on free plate eigenvalues in \cite{NF55}. Payne considered the buckling problem for the clamped plate (\cite{P55,P58}, and in conjunction with Weinberger \cite{PW57}). Kawohl, Levine, and Velte \cite{KLV} considered the clamped plate under tension and compression; they viewed the eigenvalues as functions of the tension or compression parameter and established upper and lower bounds in terms of these parameters. Analogous results for free plates are in Chapter~\ref{fundmodtens} of this dissertation.

Isoperimetric inequalities for the Laplacian and related problems were first considered more than 130 years ago, and research in this field continues today. A number of problems regarding eigenvalues of the bi-Laplacian remain open, and I hope this dissertation helps provide insight for future research on plate problems.

%% file: spectrum.tex
\chapter{The spectrum}\label{spectrum}
Our first task is to investigate the spectrum of the fourth-order operator associated with our Rayleigh quotient $Q$ in \eqref{RQ}. In this chapter we show there is only discrete spectrum, with an associated weak eigenbasis. We will then establish regularity of the eigenfunctions up to the boundary and derive the natural boundary conditions.

In this chapter we will allow $\tau$ to be any real number. We require $\Omega\subset\RR^d$ to be smoothly bounded.

\section*{The existence of the spectrum}
We consider the sesquilinear form
\begin{align*}
a(u,v) &=\int_\Omega\sum_{i,j=1}^d\overline{u_{x_ix_j}}v_{x_ix_j}+\tau(\overline{D u}\cdot D v)\,dx
\end{align*}
in $L^2(\Omega)$ with form domain $H^2(\Omega)$. Note the plate Rayleigh quotient $Q$ can be written in terms of $a$, with $Q[u]=a(u,u)/\|u\|_{L^2}^2$.

\begin{prop} \label{spect}The spectrum of the operator $A$ associated with the form $a(\cdot,\cdot)$ above consists entirely of isolated eigenvalues of finite multiplicity $\omega_0\leq\omega_1\leq\omega_2\leq\dots\to\infty$. There exists an associated set of real-valued weak eigenfunctions which is an orthonormal basis for $L^2(\Omega)$.
\end{prop}

\begin{proof}
By Cauchy-Schwarz, the form $a(\cdot,\cdot)$ is bounded, and therefore continuous, on $H^2(\Omega)$. We will show the quadratic form $a(u,u)$ is coercive; that is, for some positive constants $c_1$ and $c_2$, we have $a(u,u)+c_1\|u\|^2\geq c_2\|u\|^2_{H^2(\Omega)}$. By the boundedness of $a$ on $H^2$, this is equivalent to showing that the norm
\[
 \|u\|_a^2=a(u,u)+c_1\|u\|^2,
\]
is equivalent to $\|\cdot\|^2_{H^2(\Omega)}$, and hence $a$ is a closed quadratic form on $H^2(\Omega)$. 

Once we have coercivity and show that $H^2(\Omega)$ is compactly embedded in $L^2(\Omega)$, we conclude by a standard result (see e.g., Corollary 7.D \cite[p. 78]{SH77}) that the form $a$ has a set of weak eigenfunctions which is an orthonormal basis for $L^2(\Omega)$, and the corresponding eigenvalues are of finite multiplicity and satisfy
\begin{equation}\label{eigenvalueineq}
 \omega_0\leq\omega_1\leq \dots\leq\omega_n\rightarrow\infty\quad\text{as}\quad n\rightarrow\infty.
\end{equation}

We first show that $H^2(\Omega)$ is compactly embedded in $L^2(\Omega)$. Let $B_{R}$ be a ball of radius $R$ centered such that $\Omega\subset B_R$; thus we have $H^2(\Omega)\subset L^2(\Omega)\subset L^2(B_R)$ and can identify $L^2(\Omega)$ with a closed subspace of $L^2(B_R)$. Because $\Omega$ is smoothly bounded, we can extend all functions in $H^2(\Omega)$ to functions in $H^2_o(B_{2R})$. The extension map is linear and bounded, so $H^2(\Omega)$ can be identified with a closed subset of $H^2_o(B_{2R})$.  The space $H^2_o(B_{2R})$ is compactly embedded in $L^2(B_R)$ by the Rellich-Kondrachov Theorem; thus because the space $H^2(\Omega)$ is a closed subset of $H^2_o(B_{2R})$, we have that $H^2(\Omega)$ also compactly embedded in $L^2(B_R)$. Since $L^2(\Omega)$ is a closed subspace of $L^2(B_R)$ containing $H^2(\Omega)$, we have that $H^2(\Omega)$ is a compact subspace of $L^2(\Omega)$. (See, e.g., \cite{adams} for extension theorems and the Rellich-Kondrachov Theorem.)

We next show coercivity of the form $a$. For $\tau>0$, coercivity is easily proved:
\begin{align*}
a(u,u) + \tau \|u\|^2&\geq \|D^2u\|^2+\tau\|D u\|^2 +\tau \|u\|^2\\
&\ge \min(\tau,1)\|u\|_{H^2}^2,
\end{align*}
where all unlabeled norms are $L^2$ norms on $\Omega$. 

To prove coercivity when $\tau \leq 0$, we must somehow arrive at a positive constant in front of the $|Du|^2$ term. We cannot use Poincar\'e's inequality on the $|D^2u|$ term as this will introduce terms involving the average value of $Du$. Instead, we will exploit an interpolation inequality.

By Theorem 7.28 of \cite[p. 173]{GT}, we have that for any index $1\leq j\leq n$ and any $\epsilon>0$,
\begin{equation}
\|\partial_{x_j}u\|_{L^2(\Omega)}^2 \leq \epsilon\|u\|_{H^2(\Omega)}^2+C\epsilon^{-1}\|u\|_{L^2(\Omega)}^2
\end{equation}
with $C=C(\Omega)$ a constant. Replacing $\epsilon$ by $\epsilon/d$ and summing over $j$, we see
\begin{equation*}
\|D^2u\|_{L^2}^2 \geq \left(\frac{1}{\epsilon}-1\right)\|Du\|_{L^2}^2
-\left(\frac{C}{\epsilon^2}+1\right)\|u\|_{L^2}^2.
\end{equation*}
Fix $\delta\in(0,1)$. Let $K > 0$. Then
\begin{align*}
a(u,u)&+K\|u\|_{L^2}^2= \|D^2u\|_{L^2}^2-|\tau|\|Du\|^2_{L^2}+K\|u\|_{L^2}^2\\
&\geq (1-\delta)\|D^2u\|_{L^2}^2+\left(\frac{\delta}{\epsilon}-\delta-|\tau|\right)\|Du\|^2_{L^2}+\left(K-\frac{C\delta}{\epsilon^2}-\delta\right)\|u\|_{L^2}^2\\
&\geq \min\left\{1-\delta,\frac{\delta}{\epsilon}-\delta-|\tau|,K-\frac{C\delta}{\epsilon^2}-\delta\right\}\|u\|_{H^2},
\end{align*}
We can choose our $\epsilon$ small and our $K$ large so that the minimum is positive, which proves coercivity. For example, for $\delta=1/2$, we need to take \\$\epsilon<1/(1+2|\tau|)$ and $K>\frac{1}{2}\Big(C+1+2|\tau|\Big)$. 

We now have that the form $a$ is coercive for all $\tau\in\RR$. Now suppose $u$ is a weak eigenfunction corresponding to eigenvalue $\omega$. Because $\omega$ and $\tau$ are real-valued, by taking the complex conjugate of the weak eigenvalue equation we see that $\ubar$ is also a weak eigenfunction with the same eigenvalue. Thus the real and imaginary parts of $u$ are both eigenfunctions associated with $\omega$, and we may choose our eigenfunctions to be real-valued.
\end{proof}

Note that for any bounded region $\Omega$ and all real values of $\tau$, the constant function solves the weak eigenvalue equation with eigenvalue zero. For all nonnegative values of $\tau$, the Rayleigh quotient is nonnegative for all functions and so $0=w_0\leq w_1\leq\cdots$.  When $\tau=0$, the coordinate functions $x_1,\dots,x_d$ are also solutions with eigenvalue zero, and so the lowest eigenvalue is at least $d+1$-fold degenerate, as noted in the introduction. Taking instead $\tau>0$, the Raleigh quotient shows that the fundamental tone $\omega_1$ is positive, and so we have:
\[
 0=\omega_0<\omega_1\leq\omega_2\leq\cdots\leq\omega_n\rightarrow\infty\quad\text{as}\quad n\rightarrow\infty.
\]

\section*{Regularity}
We aim to establish regularity of the weak eigenfunctions by appealing to interior and boundary regularity theory for elliptic operators. 

\begin{prop} \label{regprop} For any $\tau \in\RR$ and smoothly bounded $\Omega$, the weak eigenfunctions of the operator $A$ are smooth on $\overline{\Omega}$. 
\end{prop}

\begin{proof}
Let $u$ be a weak eigenfunction of $A$ with associated eigenvalue $\omega$; by Proposition~\ref{spect} we have $u\in\DD(a)=H^2(\Omega)$. Then by a theorem in \cite[p 668]{Nir55}, we have $u\in H^k(\Omega)$ for every positive integer $k$. Thus we have $u\in H^k(\Omega)$ for all $k\in \mathbb{Z}^+$, and so $u\in C^\infty(\Omega)$.

Regularity on the boundary follows from global interior regularity and the Trace Theorem (see, for example, \cite[Prop 4.3, p. 286 and Prop 4.5, p. 287.]{taylor}). Thus we have $u\in C^\infty(\overline{\Omega})$, as desired.
\end{proof}

\section*{The Natural Boundary Conditions}
In this section, our goal is to derive the form of the natural boundary conditions necessarily satisfied by all weak eigenfunctions.

In the case of the free membrane, the weak eigenfunctions $u$ are smooth on $\overline{\Omega}$ and satisfy
\[
 \int_\Omega Du\cdot D\phi\,dx- \int_\Omega \mu u\phi\,dx=0
\]
for all $\phi\in H^1(\Omega)$. By the Divergence theorem, we obtain
\begin{equation}\label{ncase}
 \surfint \phi \frac{\partial u}{\partial n}\,dS-\int_\Omega \phi(\Delta u+\mu u)\,dx=0.
\end{equation}
Because $\phi\in H^1(\Omega)$ is arbitrary, we may consider those $\phi$ with compact support in $\Omega$. Then the boundary terms vanish, and in order for the remaining integral to be zero for all such $\phi$, we must have $\Delta u +\mu u=0$ almost everywhere on $\Omega$. Hence equation \eqref{ncase} says
\[
 \surfint \phi \frac{\partial u}{\partial n}\,dS=0
\]
 In order for this surface integral to be zero we find must have $\partial u/\partial n=0$ on $\partial\Omega$. Thus $u$ satisfies the eigenvalue equation $-\Delta u =\mu u$ on $\Omega$ and the Neumann boundary condition
\[
 \frac{\partial u}{\partial n}=0 \qquad\text{on $\partial\Omega$.}
\]
We will use the same approach to derive the natural boundary conditions for the free plate. The natural boundary conditions are rather complicated in higher dimensions, and so we state the two-dimensional case first. The boundary conditions in this case have been known for some time: see, for example, \cite{RW74}
\begin{prop} \label{2dimBC}(Two dimensions)
For $\Omega\subset\RR^2$, the natural boundary conditions for eigenfunctions of the free plate under tension have the form
\begin{align*}
&Mu := \frac{\partial^2 u}{\partial n^2} = 0 \\
&Vu :=\tau \frac{\partial u}{\partial n}-\frac{\partial (\Delta u)}{\partial n} - \frac{\partial }{\partial s} \left(\frac{\partial^2u}{\partial s\partial n}-K(s)\frac{\partial u}{\partial s}\right)= 0
\end{align*}
where $n$ denotes the outward unit normal derivative, $s$ the arclength, and $K$ the curvature of $\partial\Omega$.
\end{prop}
We also look at one example of the natural boundary conditions for a region with corners. Notice that an additional condition arises at the corners!
\begin{prop} (Rectangular region in two dimensions) \label{rectBC}
When $\Omega\subset\RR^2$ is a rectangular region with edges parallel to the coordinate axes, the natural boundary conditions for eigenfunctions of the free plate under tension have the form
\begin{align*}
&\frac{\partial^2 u}{\partial n^2} = 0 \quad\text{at each edge}\\
&\tau \frac{\partial u}{\partial n}-\frac{\partial^3 u}{\partial s^2\partial n}-\frac{\partial (\Delta u)}{\partial n}=0\qquad\text{on each edge}\\
&u_{xy}=0 \qquad\text{at each corner}
\end{align*}
where $n$ and $s$ indicate the normal and tangent directions.
\end{prop}
Finally, we state the natural boundary conditions for a smoothly-bounded region in higher dimensions:
\begin{prop} (General) \label{generalBC} For any smoothly bounded $\Omega$, the natural boundary conditions for eigenfunctions of the free plate under tension have the form
\begin{align*}
&Mu := \frac{\partial^2 u}{\partial n^2} = 0 &\text{on $\partial\Omega$,}\\
&Vu := \tau\frac{\partial u}{\partial n}
-\sdiv\Big(\sproj\left[(D^2u)n\right]\Big)-\frac{\partial\Delta u}{\partial n} = 0 &\text{on $\partial\Omega$,}
\end{align*}
where $n$ denotes the normal derivative and $\sdiv$ is the surface divergence. The projection $\sproj$ projects a vector $v$ at a point $x$ on $\partial\Omega$ into the tangent space of $\partial\Omega$ at $x$.
\end{prop}
\begin{proof}[Proof of Proposition~\ref{generalBC}]
Our eigenfunctions $u$ are smooth on $\overline{\Omega}$ by Proposition 2 and satisfy the weak eigenvalue equation $a(u,\phi)-\omega(u,\phi)_{L^2(\Omega)}=0$ for all $\phi\in H^2(\Omega)$. That is,
\[
 \int_\Omega \left(\sum_{i,j=1}^du_{x_ix_j}\phi_{x_ix_j}+\tau D\phi\cdot Du-\omega u\phi\right)\,dx=0.
\]

As in the membrane case, we make much use of integration by parts. Let $n$ denote the outward unit normal to the surface $\partial\Omega$. To simplify our calculations, we consider each term separately.

The gradient term only needs one use of integration by parts:
\[
\int_\Omega Du\cdot D\phi\,dx=\surfint\phi \frac{\partial u}{\partial n}\,dS -\int_\Omega\phi(\Delta u)\,dx.
\]

The Hessian term becomes:
\begin{align*}
\int_\Omega&\sum_{i,j}u_{x_ix_j}\phi_{x_ix_j}\,dx\\
&=\surfint\left(D\phi\cdot\Big((D^2u)n\Big)-\phi\frac{\partial(\Delta u)}{\partial n}\right)\,dS +\int_\Omega(\Delta^2 u)\phi\,dx,
\end{align*}
after integrating by parts twice.

We wish to transform the term involving $D\phi$ in the above surface integral using integration by parts. Because we are on $\partial\Omega$, we must treat the normal and tangential components separately. We can then use the Divergence theorem for integration on $\partial\Omega$.

We note that the surface gradient $\sgrad$ equals $D - n\partial_n$ when applied to a function (like $\phi$) that is defined on a neighborhood of the boundary. Thus $\sgrad \phi$ gives the tangential part of the Euclidean gradient vector. Hence,
\begin{align*}
\surfint &D\phi\cdot\Big((D^2u)n\Big)\,dS \\
&= \surfint \left(n\frac{\partial\phi}{\partial n}+\sgrad\phi\right)\cdot\left(n\frac{\partial^2u}{\partial n^2}+\sproj\left[(D^2u)n\right]\right)\,dS\\
&=\surfint \frac{\partial\phi}{\partial n}\frac{\partial^2u}{\partial n^2} +\Big\langle\sgrad\phi,\sproj\left[(D^2u)n\right]\Big\rangle_{\partial\Omega}\,dS\\
&= \surfint \frac{\partial\phi}{\partial n} \frac{\partial^2u}{\partial n^2}-\phi\,\sdiv\left(\sproj\left[(D^2u)n\right]\right)\,dS,
\end{align*}
by the Divergence Theorem on the surface $\partial\Omega$. Here $\langle\cdot,\cdot\rangle_{\partial\Omega}$ denotes the inner product on the tangent space to $\partial\Omega$. Recall $\sproj$ projects a vector at a point $x$ on $\partial\Omega$ onto the tangent space of $\partial\Omega$ at $x$.

Thus for $u$ an eigenfunction associated with eigenvalue $\omega$, we see
\begin{align*}
0&=\int_\Omega\phi\Big(\Delta^2 u-\tau\Delta u-\omega u\Big)\,dx\\
&\qquad+\surfint \frac{\partial\phi}{\partial n} \frac{\partial^2u}{\partial n^2}+\phi\left(\tau\frac{\partial u}{\partial n}-\frac{\partial\Delta u}{\partial n}-\sdiv\Big(\sproj\left[(D^2u)n\right]\Big) \right)\,dS.
\end{align*}

As in the membrane case, this identity must hold for all $\phi\in H^2(\Omega)$. If we take any compactly supported $\phi$, then the volume integral must vanish; because $\phi$ is arbitrary, we must therefore have $\Delta^2 u-\tau\Delta u-\omega u=0$ everywhere. Similarly, the terms multiplied by $\phi$ and $\partial\phi/\partial n$ must vanish on the boundary. Collecting these results, we obtain the eigenvalue equation \eqref{maineq} and natural boundary conditions of Proposition~\ref{generalBC}.
\end{proof}

\begin{proof}[Proof of Proposition~\ref{2dimBC}]
Here $d=2$; take rectangular coordinates $(x,y)$. We parametrize $\partial\Omega$ by arclength $s$ and define coordinates $(n,s)$, with $n$ the normal distance from $\partial\Omega$, taken to be positive outside $\Omega$. Write $\nhat(s)$ and $\that(s)$ for the outward unit normal and unit tangent vectors to the boundary. Then $\sproj\left[f_1\nhat+f_2\that\right]=f_2\that$ and the operators $\sdiv$ and $\sgrad$ both simply take the derivative with respect to arclength $s$. That is, for a scalar function $f(s)$, and taking $t(s)$ to be the tangent vector to the surface, we have
\[
 \sgrad f(s)=f'(s) \qquad\text{and}\qquad \sdiv(f(s)\that(s))=f'(s).
\]
 and so we may write
\[
 \sdiv\Big(\sproj\left[(D^2u)n\right]\Big)= \frac{\partial}{\partial s}t^\TT(D^2u)n.
\]
The tangent line to $\partial\Omega$ at the point $(0,s)$ in our new coordinates forms an angle $\alpha=\alpha(s)$ with the $x$-axis (see
\cite[p. 230] {RW74}); the curvature of $\partial\Omega$ is given by $K(s)=\alpha^\prime(s)$. Then in rectangular coordinates, the unit tangent vector is $(\cos\alpha,\sin\alpha)$, and the outward unit normal is $(\sin\alpha,-\cos\alpha)$. Thus we have
\[
 \frac{\partial}{\partial s}t^\TT(D^2u)s =\partial_s\Big(\sin\alpha\cos\alpha(u_{xx}-u_{yy})+(\sin^2\alpha-\cos^2\alpha)u_{xy}\Big).
\]
By \cite[p. 233]{RW74}, on $\partial\Omega$ under our change of coordinates, we have
\begin{align*}
 u_{xx}&=u_{nn}\sin^2\alpha+u_{ss}\cos^2\alpha+2u_{ns}\sin\alpha\cos\alpha+Ku_n\cos^2\alpha-2Ku_s\sin\alpha\cos\alpha\\
 u_{yy}&=u_{nn}\cos^2\alpha+u_{ss}\sin^2\alpha-2u_{ns}\sin\alpha\cos\alpha+Ku_n\sin^2\alpha+2Ku_s\sin\alpha\cos\alpha\\
 u_{xy}&=-u_{nn}\cos\alpha\sin\alpha+u_{ss}\cos\alpha\sin\alpha+u_{ns}(\sin^2\alpha-\cos^2\alpha)\\
 &\qquad+Ku_n\cos\alpha\sin\alpha-Ku_s(\sin^2\alpha-\cos^2\alpha).
\end{align*}
So after simplification,
\[
\frac{\partial}{\partial t}[n^T (D^2u)t]=\frac{\partial}{\partial s}\left(\frac{\partial^2u}{\partial s\partial n}-K(s)\frac{\partial u}{\partial s}\right).
\]
This together with the results of Proposition~\ref{generalBC} yields the form of $Vu$ given in Proposition~\ref{2dimBC}. $Mu$ is unchanged, and so this completes the proof.
\end{proof}

\begin{proof}[Proof of Proposition~\ref{rectBC}]
Our previous findings do not completely apply because $\partial\Omega$ has corners, although our argument proceeds similarly. For convenience of notation, we will take $\Omega$ to be the square $[0,1]^2$.

The Hessian term gives us a condition at the corners. In particular, after integrating by parts twice, we have:
\begin{align*}
\int_\Omega &u_{xx}\phi_{xx}+2u_{xy}\phi_{xy}+u_{yy}\phi_{yy}\,dA\\
&=\int_\Omega \phi\Big(u_{xxxx}+2u_{xxyy}+u_{yyyy}\Big)\,dA\\
&\quad+\int_{0}^{1}\Big(u_{xx}\phi_x-u_{xxx}\phi+u_{xy}\phi_y-u_{xyy}\phi\Big)\Big|_{x=0}^{x=1}\,dy\\
&\quad+\int_{0}^{1}\Big(u_{yy}\phi_y-u_{yyy}\phi+u_{xy}\phi_x-u_{xxy}\phi\Big)\Big|_{y=0}^{y=1}\,dx.
\end{align*}
Since
\begin{align*}
 \int_0^1u_{xy}\phi_y\,dy=u_{xy}\phi\Big|_{y=0}^{y=1}-\int_0^1u_{xyy}\phi\,dy
\end{align*}
and 
\begin{align*}
 \int_0^1u_{xy}\phi_x\,dx=u_{xy}\phi\Big|_{x=0}^{x=1}-\int_0^1u_{xxy}\phi\,dx
\end{align*}
we obtain
\begin{align*}
 \int_\Omega &u_{xx}\phi_{xx}+2u_{xy}\phi_{xy}+u_{yy}\phi_{yy}\,dA\\
&=\int_\Omega \phi\Big(u_{xxxx}+2u_{xxyy}+u_{yyyy}\Big)\,dA\\
&\quad+\int_{0}^{1}\Big(u_{xx}\phi_x-\phi(2u_{xyy}+u_{xxx})\Big)\Big|_{x=0}^{x=1}\,dy\\
&\quad+\int_{0}^{1}\Big(u_{yy}\phi_y-\phi(2u_{xxy}+u_{yyy})\Big)\Big|_{y=0}^{y=1}\,dx\\
&\qquad+2u_{xy}\Big|_{x=0}^{x=1}\Big|_{y=0}^{y=1}.
\end{align*}

Because the Divergence Theorem does apply to regions with piecewise-smooth boundaries, the gradient term is the same as in the smooth-boundary case. The final term above is the only term that depends only on the behavior of $u$ and $\phi$ at the corners; arguing as before, we obtain the eigenvalue equation and natural boundary conditions, with the additional condition
\[
0=u_{xy}\phi\Big|_{x=0}^{1}\Big|_{y=0}^{1}.
\]
That is, we must have $u_{xy}=0$ at the corners.\end{proof}

\section*{Example: natural boundary conditions on the ball}
When $\Omega$ is a ball, we can simplify the general boundary conditions.

\begin{prop}\label{ballBC} (Ball) The natural boundary conditions in the case $\Omega=\BB(R)$, the ball of radius $R$, are
\begin{align}
&Mu := u_{rr} = 0 &\text{at $r=R$,}\label{BCb1}\\
&Vu := \tau u_r-\frac{1}{r^2}\Delta_S\left(u_r-\frac{u}{r}\right)-(\Delta u)_r = 0&\text{at $r=R$.}\label{BCb2}
\end{align}
\end{prop}

\begin{proof}
When $\Omega$ is a ball, the normal vector to the surface at a point $x$ is $n = x/R$. Then the $i$th component of $(D^2u)n$ is given by
\[
 \sum_{j=1}^du_{x_ix_j}\frac{x_j}{R}
\]
and can be rewritten as
\[
 \frac{1}{R}\frac{\partial}{\partial x_i}\left(\sum_{j=1}^du_{x_j}x_j-u\right).
\]
Therefore,
\[
 (D^2u)n=D\left(Du\cdot\frac{x}{R}-\frac{u}{R}\right).
\]
Then the projection $\sproj$ takes the tangential component of the above gradient vector, and so
\[
 \sproj\left[(D^2u)n\right] =\grad_{\partial\BB(R)}\left(Du\cdot\frac{x}{R}-\frac{u}{R}\right).
\]
We know $\sdiv\sgrad=\Delta_{\partial\Omega}$ by definition. For the ball of radius $R$, we have $\Delta_{\partial\BB(R)}=\frac{1}{R^2}\Delta_S$. The operator $\Delta_S$ is the spherical Laplacian, consisting of the angular part of the Laplacian. It satisfies the identity $\Delta=\frac{\partial^2}{\partial r^2}+\frac{d-1}{r}\frac{\partial}{\partial r}+\frac{1}{r^2}\Delta_S$.

Thus
\begin{align*}
 \sdiv\sproj\left[(D^2u)n\right]&=\Delta_{\partial\BB(R)}\left(Du\cdot\frac{x}{R}-\frac{u}{R}\right)\\
&=\Delta_{\partial\BB(R)}\left(u_r-\frac{u}{R}\right),
\end{align*}
by noting that $Du\cdot n=\partial u/\partial r$.  The boundary conditions of Proposition~\ref{generalBC} then simplify to \eqref{BCb1} and \eqref{BCb2}, as desired.\end{proof}

%% file: fundmodtens.tex
\chapter{The fundamental tone as a function of tension}\label{fundmodtens}
Fix the smoothly bounded domain $\Omega$. We will estimate how the fundamental tone $\omega_1=\omega_1(\tau)$ depends on the tension parameter $\tau$, for use in the proof of Theorem~\ref{thm1}. We will also study in this chapter the behavior of $\omega_1$ in the extreme cases as $\tau\to 0$ and $\tau\to\infty$. 

First we note that the Rayleigh quotient \eqref{RQ} is linear and increasing as a function of $\tau$. Our eigenvalue $\omega_1(\tau)$ is the infimum of $Q[u]$ over $u\in H^2(\Omega)$ with $\int_\Omega u\,dx=0$, and thus $\omega_1(\tau)$ is itself a concave, increasing function of $\tau$.

Next, we will prove $\omega_1(\tau)/\tau$ is bounded above and below for all $\tau>0$. Recall $\mu_1$ is the fundamental tone of the free membrane.

\begin{lemma}\label{wbounds} For all $\tau\geq 0$ we have
 \begin{equation}\label{lem31}
  \tau\mu_1 \leq\omega_1(\tau)\leq \tau \frac{|\Omega|d}{\int_\Omega |x-\bar{x}|^2 \, dx},
 \end{equation}
where $\bar{x}=\int_\Omega x\,dx/|\Omega|$ is the center of mass of $\Omega$. In particular, when $\Omega$ is the unit ball,
 \begin{equation} 
  \tau\mu_1 \leq \omega_1(\tau) \leq \tau(d+2) .\label{omegabounds}
 \end{equation}
 Furthermore, the upper bounds in \eqref{lem31} and \eqref{omegabounds} hold for all $\tau\in\RR$.
\end{lemma}
These bounds are illustrated in Figure~\ref{Lemma31bounds}.
\begin{proof}
To establish the upper bound, take the coordinate functions as trial functions: $u_k=x_k-\bar{x}_k$, for $k=1,\dots,d$. Note $\int_\Omega u_k\,dx=0$ by definition of center of mass, so the $u_k$ are valid trial functions. All second derivatives of the $u_k$ are zero, so we have
\[
 \omega_1(\tau) \leq Q[u_k] =\frac{\int_\Omega\tau|Du_k|^2\,dx}{\int_\Omega u_k^2\,dx}
	=\tau\frac{\int_\Omega 1\,dx}{\int_\Omega (x_k-\bar{x}_k)^2\,dx}.
\]
Clearing the denominator and summing over all indices $k$, we obtain
\[
 \omega_1(\tau)\int_\Omega |x-\bar{x}|^2\,dx \leq \tau |\Omega| d,
\]
which is the desired upper bound. When $\Omega$ is the unit ball, note $\int_\Omega |x|^2\,dx=|\Omega|d/(d+2)$.

Now we treat the lower bound. Let $u\in H^2(\Omega)$ with $\int_\Omega u\,dx=0$. Then
\[
 Q[u]\geq\frac{\tau\int_\Omega |Du|^2\, dx}{\int_\Omega u^2\,dx} \geq \tau \mu_1
\]
by the variational characterization of $\mu_1$. Taking the infimum over all trial functions $u$ for the plate yields $ \omega_1(\tau)\geq \tau\mu_1$.
\end{proof}
\begin{figure}[t!]
\begin{center}
\includegraphics{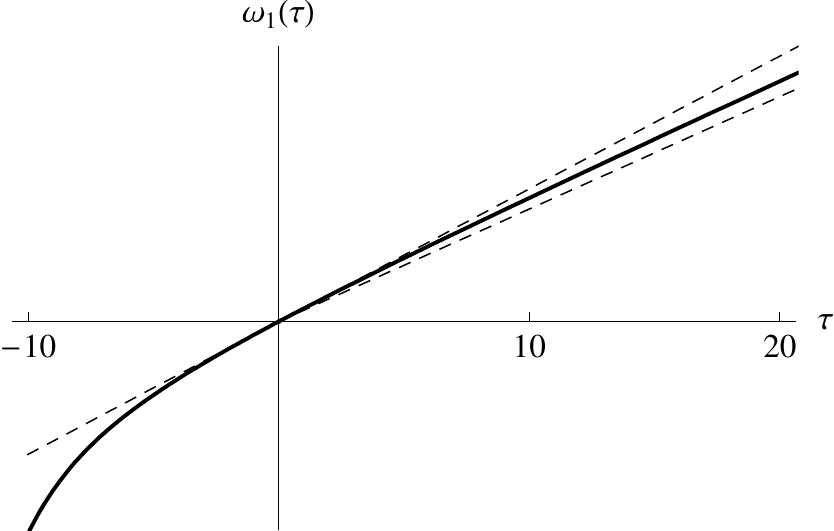}
\caption{The fundamental tone of the disk (solid curve) together with the linear bounds from Lemma~3.1 (dashed lines).}\label{Lemma31bounds}
\end{center}
\end{figure}

Note that Payne \cite{P56} proved linear bounds for eigenvalues of the \emph{clamped} plate under tension. Kawohl, Levine, and Velte \cite{KLV} investigated the sums of the first $d$ eigenvalues as functions of parameters for the clamped plate under tension and compression.

We can also prove another linear upper bound on $\omega_1$, which is just a constant plus the lower bound in Lemma~\ref{wbounds}.
\begin{lemma} \label{wbounds2} For all $\tau\in\RR$,
\[
\omega_1 \leq C(\Omega)+\tau\mu_1,
\]
where the value
\[
 C(\Omega)=\frac{\int_\Omega|D^2v|^2\,dx}{\int_\Omega v^2\,dx}
\]
is given explicitly in terms of the fundamental mode $v$ of the free membrane on $\Omega$.
\end{lemma}
\begin{proof}
 Let $v$ be a fundamental mode of the membrane with $\Delta v=-\mu_1 v$ and $\int_\Omega v\,dx=0$; the membrane boundary condition is $\partial u/\partial n = 0$ on $\partial\Omega$. Then by the variational characterization of eigenvalues,
\begin{align*}
 \omega_1(\tau)\leq Q[v]&= C(\Omega) +\tau Q_\text{M}[v]=C(\Omega)+\tau\mu_1,
\end{align*}
as desired.
\end{proof}

\begin{figure}[h!]
\begin{center}
\includegraphics{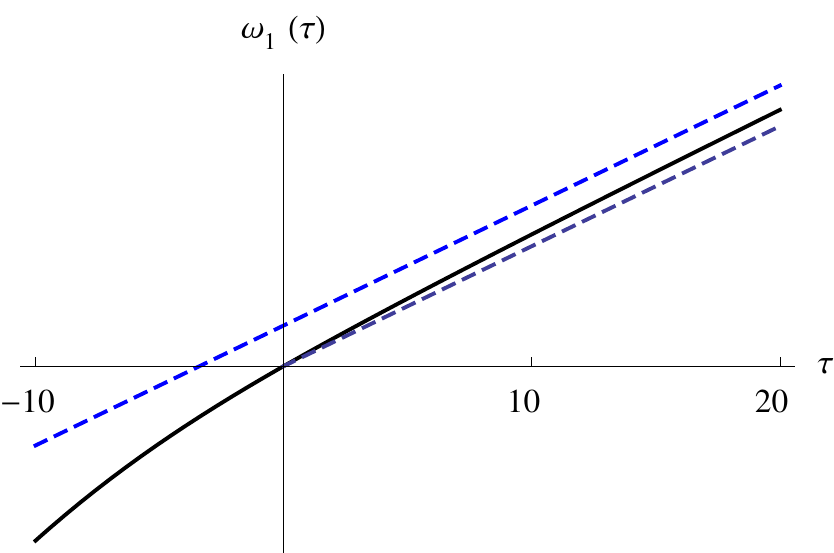}
\caption{The fundamental tone of the disk (solid curve) together with the upper bound of Lemma~3.2 (top dashed line) and the lower bound of Lemma~3.1 (bottom dashed line).}
\end{center}
\end{figure}

\subsection*{Infinite tension limit}

A plate behaves like a membrane as the flexural rigidity tends to zero, that is, as $\tau=(\text{tension}/\text{flexural rigidity})$ tends to infinity. For the fundamental tone, that means:
\begin{cor} For the fundamental tone of the free plate,
\[
\frac{\omega_1(\tau)}{\tau}\to\mu_1\qquad\text{as $\tau\to\infty$.}
\]
\end{cor}
\begin{proof}
By Lemmas~\ref{wbounds} and~\ref{wbounds2}, we have
\[
\mu_1\leq \frac{\omega_1(\tau)}{\tau}\leq\mu_1+\frac{C}{\tau}.
\]
Let $\tau\to\infty$.
\end{proof}

The eigenfunctions should converge as $\tau\to\infty$ to the eigenfunctions of the free membrane problem. Proving this for all eigenfunctions seems to require a singular perturbation approach, which has been carried out for the clamped plate in \cite{deGroen}, but we will not need any such facts for our work. For the convergence of the fundamental tone of the clamped plate to the first fixed membrane eigenvalue, see \cite{KLV}.

\subsection*{Vanishing tension limit; moment of inertia}
At $\tau=0$, the lowest eigenvalue is zero and has multiplicity $d+1$ since $Q[u]=0$ for any linear function $u$. We will establish a relationship between the (scalar) moment of inertia $I_\Omega$ of our region $\Omega$ and the derivatives at $\tau=0$ of the first $d$ nontrivial eigenvalues $\omega_1(\tau)$, $\dots$, $\omega_d(\tau)$.
\begin{lemma}\label{minertia} For $\tau>0$, we have 
\[
|\Omega|\sum_{j=1}^d\frac{\tau}{\omega_j(\tau)}\geq I_\Omega.
\]
\end{lemma}
We will not need this result later, except as motivation for some conjectures. 

 If $\omega_j(\tau)$ is differentiable from the right at $t=0$ then we deduce the following bound involving the derivatives of eigenvalues with respect to $\tau$:
\[
|\Omega|\sum_{j=1}^d\frac{1}{\omega_j^\prime(0)}\geq I_\Omega.
\]
\begin{proof} We assume our plate has its center of mass at the origin, so that the scalar moment of inertia $I_\Omega$ may be expressed as
\[
 I_\Omega=|\Omega|\int_\Omega r^2\,dx \qquad\text{or}\qquad  I_\Omega=\trace{M_\Omega},
\]
where $M_\Omega=\int_\Omega xx^T\,dx$ is the moment matrix whose entries are given by $m_{ij}=\int_\Omega x_ix_j\,dx$. 

Take $a(.,.)$ to be the sesquilinear form from Chapter 2. As in \cite[p. 99]{Bandle}, we define the inverse trace $\trace^{-1}[L_m]$ of the $m$-dimensional space \\$L_m\subset \{v\in H^2(\Omega): \int_\Omega v\,dx=0\}$ by
\[
 \trace^{-1}[L_m]:=\sum_{j=1}^mQ[w_j]^{-1}=\sum_{j=1}^m\int_\Omega w_j^2\,dx,
\]
where the $w_j$ form a basis of $L_m$ satisfying the orthonormality condition\\ $a(w_j,w_k)=\delta_{jk}$.  Then we have, again by \cite{Bandle}, the variational characterization
\[
 \sum_{j=1}^d\frac{1}{\omega_j(\tau)} = \max\left\{\trace^{-1}[L_d] : L_d\subset H^2(\Omega)\quad \text{with}\quad \int_\Omega v\,dx=0~ \forall v\in L_d\right\}.
\]
Considering the coordinate functions $x_j$, we see $a(x_j,x_k)=\tau|\Omega|\delta_{jk}$ and\\ $\int_\Omega x_j=0$, so we may take $w_j=x_j/\sqrt{\tau|\Omega|}$, and our variational characterization gives us
\begin{equation}
  \tau|\Omega|\sum_{j=1}^d\frac{1}{\omega_j(\tau)} \geq\sum_{j=1}^d\int_\Omega x_j^2\,dx.\label{inertiabound}
\end{equation}
The righthand side is simply the trace of $M_\Omega$ and hence is equal to our scalar moment of inertia $I_\Omega$. \end{proof}

%% file: bessel.tex
\chapter{Ultraspherical Bessel functions}\label{bessel}
We must examine properties of $d$-dimensional ultraspherical Bessel functions, for they provide the eigenfunctions on the unit ball for dimensions 2 and higher, in the next chapter. For more information on Bessel functions, see \cite[p.358-389]{AShandbook}. For more information on spherical and ultraspherical Bessel functions, see  \cite[p.437-455]{AShandbook} ($d=3$ only) and \cite{LS94} (all $d\geq2$).

\section*{Definitions}
The Bessel function $J_\nu(z)$ is defined by the power series
\[
 J_\nu(z)=\sum_{k=0}^\infty\frac{(-1)^k}{k!\,\Gamma(\nu+k+1)}\left(\frac{z}{2}\right)^{2k+\nu},
\]
and is hence analytic. This function solves the Bessel differential equation, $z^2w''+zw'+(z^2-\nu^2)w = 0$. 

For higher dimensions, we need to consider spherical ($d=3$) and ultraspherical ($d\geq 4$) Bessel functions $j_l(z)$, defined by:
\begin{align*}
j_{l}(z) &= z^{-s}J_{s+l}(z)\\
\text{with}\quad s &= \frac{d-2}{2}.
\end{align*}
Such functions solve the equation
\begin{equation}
z^2w''+(d-1)zw'+\Big(z^2-l(l+d-2)\Big)w = 0.\label{besseleqn}
\end{equation}

Analogously, the modified Bessel function $I_\nu(z)$ is given by the power series
\[
 I_\nu(z)=\sum_{k=0}^\infty\frac{1}{k!\,\Gamma(\nu+k+1)}\left(\frac{z}{2}\right)^{2k+\nu},
\]
and solves the modified Bessel equation $z^2w''+zw'-(z^2+\nu^2)w = 0$. We define the higher-dimensional analog $i_{l}(z)$ as follows:
\begin{align*}
i_{l}(z) &= z^{-s}I_{s+l}(z)\\
\text{with}\quad s &= \frac{d-2}{2}.
\end{align*}
Such functions solve the equation
\begin{equation}
z^2w''+(d-1)zw'-\Big(z^2+l(l+d-2)\Big)w = 0. \label{modbesseleqn}
\end{equation}

\subsection*{Recurrence Relations and power series}
The Bessel functions $J_\nu$ and $I_\nu$ have a number of useful recurrence relations. Those listed below are taken from \cite[p. 361, 376]{AShandbook}.
\begin{align*}
 \frac{2\nu}{z}J_\nu(z)&=J_{\nu-1}(z)+J_{\nu+1}(z)\\
 J_\nu^\prime(z)&=\frac{\nu}{z}J_\nu(z)-J_{\nu+1}(z)\\
 &=J_{\nu-1}(z)-\frac{\nu}{z}J_{\nu}(z)\\
 \frac{2\nu}{z}I_\nu(z)&=I_{\nu-1}(z)-I_{\nu+1}(z)\\
 I_\nu^\prime(z)&=\frac{\nu}{z}I_\nu(z)+I_{\nu+1}(z)\\
 &=I_{\nu-1}(z)-\frac{\nu}{z}I_{\nu}(z)
\end{align*}
From these we also have recurrence relations involving second derivatives:
\begin{align*}
J_\nu^{\prime\prime}(z)&=\left(\frac{\nu^2-\nu}{z^2}-1\right)J_\nu(z)+\frac{1}{z}J_{\nu+1}(z)\\
I_\nu^{\prime\prime}(z)&=\left(\frac{\nu^2-\nu}{z^2}+1\right)I_\nu(z)-\frac{1}{z}I_{\nu+1}(z)
\end{align*}

The ultraspherical Bessel functions have similar recurrence relations, all of which follow from the definition and application of the corresponding ordinary Bessel recurrence relations:
\begin{align}
\frac{d-2+2l}{z}j_l(z) &= j_{l-1}(z)+j_{l+1}(z)\label{j1}\\
j_l^\prime(z) &= \frac{l}{z}j_l(z)-j_{l+1}(z)\label{j2}\\
&=j_{l-1}(z)-\frac{l+d-2}{z}j_l(z)\label{j3}\\
\frac{d-2+2l}{z}i_l(z) &= i_{l-1}(z)-i_{l+1}(z)\label{i1}\\
i_l^\prime(z) &= \frac{l}{z}i_l(z)+i_{l+1}(z)\label{i2}\\
&=i_{l-1}(z)-\frac{l+d-2}{z}i_l(z)\label{i3}
\end{align}
Note that if we take $d=2$, each of these simplifies to the corresponding relation for Bessel functions.

We also have recurrence relations for the second derivatives:
\begin{align}
j_l^{\prime\prime}(z) &= \left(\frac{l^2-l}{z^2}-1\right)j_l(z)+\frac{d-1}{z}j_{l+1}(z)\label{j4}\\
i_l^{\prime\prime}(z) &= \left(\frac{l^2-l}{z^2}+1\right)i_l(z)-\frac{d-1}{z}i_{l+1}(z)\label{i4}.
\end{align}
Again, when $d=2$ each recurrence relation simplifies to its two-dimensional analog.

We may also write a power series for the ultraspherical Bessel functions $j_l(z)$ and $i_l(z)$ using the series for the corresponding $J_{s+l}$ and $I_{s+l}$:
\begin{align}
 j_l(z) &= \sum_{k=0}^\infty\frac{(-1)^k\,2^{1-d/2}}{k!\,\Gamma(k+\frac{d}{2}+l)}\left(\frac{z}{2}\right)^{2k+l}\label{jseries}\\
i_l(z) &=\sum_{k=0}^\infty\frac{2^{1-d/2}}{k!\,\Gamma(k+\frac{d}{2}+l)}\left(\frac{z}{2}\right)^{2k+l} .\label{iseries}
\end{align}
By examining the power series \eqref{iseries}, it is immediate that $i_l(z)$ and its derivatives are all positive on $(0,\infty)$. Since the terms of the power series for $j_l$ and $i_l$ are the same up to a sign, we also have that the derivatives of $j_l$ are dominated by those of $i_l$:
\begin{equation}
\Big|j_l^{(m)}(z)\Big| \leq i_l^{(m)}(z)\qquad \text{for $m\geq0$, $z\geq 0$,}
\end{equation}
with equality only at $z=0$.

\subsection*{Other needed facts}
To prove our main result, we will need several facts about Bessel functions and their derivatives. We begin with a result on the zeroes of the $j_l^\prime(z)$.

\begin{prop}[L. Lorch and P. Szego, \cite{LS94}]\label{propLS}
Let $p_{l,k}$ denote the $k$th positive zero of $j_l^\prime(z)$. Then for $d\geq3$ and $l\geq 1$,
\[
 \frac{l(d+2l)(d+2l+2)}{d+4l+2}<\left(p_{l,1}\right)^2<l(d+2l).
\]
In particular, for $p_{1,1}$ the first zero of $j_1^\prime$, we deduce
\[
d< p_{1,1}^2< d+2.
\]
This inequality holds for all $d\geq2$.
\end{prop}
Recall $s=(d-2)/2$. 
\begin{lemma} \label{fact1} The functions $j_l$ and $J_{s+l}$ have the same sign. In particular, for $l= 1,\dots,5$ and any $d\geq 2$, we have $j_l(z)>0$ for $z\leq p_{1,1}$.
\end{lemma}
\begin{proof} The first statement is immediate from the definition of the ultraspherical Bessel functions. For the second statement, we appeal to established facts of Bessel functions. 
If we write $j_{l,1}$ for the first nontrivial zero of the Bessel function $J_l(z)$. It is a well-known fact that $J_l(z)$ is positive on $(0,j_{l,1})$ and the zeroes $j_{l,1}$ are increasing in $l$ for $l\geq 1$. Because $J_1(z)=0$ at $z=0$ and $j_{1,1}$ with no zeroes between, we have the same for $j_1(z)$ and thus the first root of $j_1^\prime(z)$, $p_{1,1}$, lies between $0$ and $j_{1,1}$. Therefore for any $d\geq 2$ and any $l\geq1$, we have $J_l(z)>0$ and hence $j_l(z)>0$ on $(0,p_{1,1}]$.
\end{proof}

\begin{lemma} \label{fact1.5} We have $j_1^\prime>0$ on $(0,p_{1,1})$.
 \end{lemma}
\begin{proof}
 This follows from the observation that $j_1(z)>0$ on $(0,j_{1,1})$ and the definition of $p_{1,1}$.
\end{proof}

\begin{lemma}\label{fact2} We have $j_2'>0$ on $(0,p_{1,1}]$.
\end{lemma}
\begin{proof} Let $p_{2,1}$ denote the first zero of $j_2'$. By Proposition~\ref{propLS}, $p_{1,1}^2<d+2$ and
\[
p_{2,1}^2 >\frac{2(d+4)(d+6)}{d+10}.
\]
Then $p_{2,1}^2-p_{1,1}^2>(d^2+8d+28)/(10+d)>0$, so that $j_2^\prime>0$ on $(0,p_{1,1}]$. 
\end{proof}

\begin{lemma}\label{fact3} We have $j_1'' < 0$ on $(0,p_{1,1}]$.
\end{lemma}
\begin{proof}
We see that
\begin{align*}
j_1^{\prime\prime}(z) &= \frac{d-1}{z}j_2(z)-j_1(z) &&\text{by \eqref{j4}}\\
&=-\frac{1}{z}j_2(z)-j_2^\prime(z)&&\text{by \eqref{j3} with $l=2$.}
\end{align*}
Since both $j_2$, $j_2^\prime$ are positive on $(0,p_{1,1}]$ by the previous lemmas, we obtain $j_1^{\prime\prime}$ on that same interval.
\end{proof}

\begin{lemma}\label{fact4} We have $j_1^{(4)}>0$ on $(0,p_{1,1}]$.
\end{lemma}
\begin{proof} 
We have by \eqref{j4} that
\[
 j_1^{\prime\prime}(z)=-j_1(z)+\frac{d-1}{z}j_2(z),
\]
and so
\begin{align}
 j_1^{(4)}&=-j_1^{\prime\prime}(z)+\frac{d-1}{z}j_2^{\prime\prime}(z)-\frac{2(d-1)}{z^2}j_2^\prime(z)+\frac{2(d-1)}{z^3}j_2(z) \nonumber \\
&=j_1(z)-\frac{2(d-1)}{z}j_2(z)+\frac{d^2-1}{z^2}j_3(z) \label{fact4eq1}
\end{align}
by \eqref{j4} with $l=1$ and $l=2$, and \eqref{j2} with $l=2$. 
When $d=2$, this becomes
\begin{equation}
  j_1^{(4)}(z)=\left(1-\frac{3}{z^2}\right)j_1(z)+\left(\frac{12}{z^3}-\frac{2}{z}\right)j_2(z) \label{d2case}
\end{equation}
by \eqref{j1} with $l=2$. For any $d$, \eqref{fact4eq1} gives us
\begin{align}
 j_1^{(4)}(z)&=\frac{4-d}{z}j_2(z)+\left(\frac{d^2-1}{z^2}-1\right)j_3(z) \
&\quad\text{by \eqref{j1} with $l=2$}\label{dsmall}\\
&=\left(\frac{15}{z^2}-1\right)j_3(z)+\frac{d-4}{z}j_4(z) \\
&\quad\text{by \eqref{j1} with $l=3$}\label{dbigger1}\\
&=\left(\frac{15(d+6)}{z^3}-\frac{10}{z}\right)j_4(z)+\left(1-\frac{15}{z^2}\right)j_5(z)\\
&\quad\text{by \eqref{j1} with $l=4$}\label{dbigger2}.
\end{align}

When $d=2$, then the first term of \eqref{dsmall} is nonnegative on $(0,p_{1,1}]$ by Lemma~\ref{fact2}. The function $j_3$ is positive on $(0,p_{1,1}]$ by Lemma~\ref{fact1}; note that since $d=2$, we have $j_3(z)=J_3(z)$. Thus we have $j_1^{(4)}(z)>0$ when $z\in(0,\sqrt{3}]\cap(0,p_{1,1}]$. However, $p_{1,1}>\sqrt{3}$, so we have only established positivity on $(0,\sqrt{3}]$.

To establish positivity on $(\sqrt{3},p_{1,1}]$ we turn to \eqref{d2case}. The first term is certainly positive on $(\sqrt{3},p_{1,1}]$. The second term is positive when both $J_2>0$ and $z<\sqrt{6}$. Because $p_{1,1}\approx 1.84$ for $d=2$, we have $p_{1,1}<6$ and we are done.

When $d=3$ and $d=4$, we again examine \eqref{dsmall}. Then Lemma~\ref{fact1} together with the argument above give us $j_1^{(4)}>0$ on $(0,\sqrt{d^2-1})\cap(0,p_{1,1}]$. By Proposition~\ref{propLS} we have $p_{1,1}<\sqrt{d+2}$, which for $d=3$ and 4 is less than $\sqrt{d^2=1}$, thus proving the lemma for these $d$.

For dimensions $d\geq 5$, we turn to \eqref{dbigger1}.  The second term is positive on $(o,p_{1,1}]$ for all $d>4$ by Lemma~\ref{fact1}. Since $p_{1,1}<\sqrt{d+2}$ and $\sqrt{d+2}\leq \sqrt{15}$ for $d\leq 13$, we conclude $j_1^{(4)}(z)>0$ on $(0,p_{1,1}]$ for $d=5,\dots,13$. 

Finally, suppose $d\geq 14$ and $z\in(0,p_{1,1}]$. If $z\in(0,\sqrt{15}]$, then $j_1^{(4)}(z)>0$ as above. If $z>\sqrt{15}$, then we examine \eqref{dbigger2}. Here the first term is nonnegative on $[\sqrt{15},p_{1,1}]$. The non-Bessel factor of the second term is positive on $(0,\sqrt{\frac{3}{2}(d+6)}]$ and hence on $(0,p_{1,1}]$.
\end{proof}

Let $d_k$ denote the coefficients of the series expansion for $i_1^{\prime\prime}(z)$, so that
\[
j_1^{\prime\prime}(z) = \sum_{k=1}^\infty (-1)^k d_k z^{2k-1} \quad \text{and} \quad i_1^{\prime\prime}(z) = \sum_{k=1}^\infty d_k z^{2k-1}
\]
by \eqref{jseries} and \eqref{iseries}, where 
\begin{align*}
d_k&=\frac{2k+1}{(k-1)!\Gamma(k+1+d/2)}2^{1-2k-d/2}.
\end{align*}

\begin{lemma}\label{ijbounds} We have the following bounds:
 \begin{align*}
 -d_1 z+d_2 z^3&\geq j_1^{\prime\prime}(z)&\text{for all $z\in\Big[0,\sqrt{3(d+2)/(d+5)}\Big]$,}\\
d_1 z+\frac{6}{5}d_2 z^3&\geq i_1^{\prime\prime}(z)&\text{for all $z\in\Big[0,\sqrt{3}\Big]$.}
 \end{align*}
\end{lemma}
\begin{proof} 
Let 
\[
c_k:=\frac{d_{k+1}}{d_k} = \frac{2k+3}{2k(2k+1)(2k+d+2)}.
\]
It is easy to show that $c_k$ is decreasing for $k \geq 1$.

We use the series expansion to first prove the following upper bound on $j_1^{\prime\prime}(z)$ for $z\geq 0$:
\begin{align*}
(-d_1 z+d_2 z^3)-j_1^{\prime\prime}(z) &= \sum_{k=3}^\infty (-1)^{k+1} d_k z^{2k-1} \\
 &=\sum_{\substack{k=3 \\ k \text{odd}}}^\infty(1-c_k z^2)d_k z^{2k-1}\\
 &\geq (1-c_1 z^2)\sum_{\substack{k=3 \\ k \text{odd}}}^\infty d_k z^{2k-1},
\end{align*}
since $c_k$ is decreasing in $k$. Hence $(-d_1 z+d_2 z^3)-j_1^{\prime\prime}(z)\geq0$ when\\ $0\leq z\leq 1/\sqrt{c_1}=\sqrt{6(d+4)/5}$, which is a larger range even than claimed in the first estimate in the lemma.

For $i_1^{\prime\prime}(z)$ we must take a slightly different approach. We will show that on $[0,\sqrt{3}]$,
\[
\frac{1}{5} d_2 z^3 \geq \sum_{k=3}^\infty d_k z^{2k-1},
\]
 and thus
\begin{equation}
 d_1 z+\frac{6}{5}d_2 z^3 \geq i_1^{\prime\prime}(z). \label{ibound}
 \end{equation}
On $[0,\sqrt{3}]$, note that
\begin{align*}
\sum_{k=3}^\infty d_k z^{2k-1} &= \sum_{k=3}^\infty  \frac{2k+1}{(k-1)!\Gamma(k+1+d/2)}2^{-d/2}\left(\frac{z}{2}\right)^{2k-1}\\
&\leq 2^{-d/2}\left(\frac{z}{2}\right)^3\sum_{k=3}^\infty \frac{2k+1}{(k-1)!(k+d/2)\Gamma(k+d/2)}\left(\frac{\sqrt{3}}{2}\right)^{2k-4} \\
&\qquad\text{since $z\leq\sqrt{3}$}\\
  &\leq \frac{2^{-d/2}}{\Gamma(3+d/2)} \left(\frac{z}{2}\right)^3 \sum_{k=3}^\infty \frac{2k+1}{(k-1)!(k+d/2)}\left(\frac{3}{4}\right)^{k-2}\\
&\qquad\text{since $\Gamma(z)\geq1$ and is increasing on $[2,\infty)$,}\\
  &\leq \frac{d_2}{5}z^3 \sum_{k=2}^\infty \frac{2}{k!} \left(\frac{3}{4}\right)^{k-1}\\
&\qquad\text{by the definition of $d_2$ and taking $k\mapsto k+1$}\\
  &=\frac{8}{15}d_2 z^3\left(e^{3/4}-1-3/4\right)\\
&\qquad\text{by the power series for $e^x$}\\
&\leq \frac{1}{5} d_2 z^3.
\end{align*}
Thus we have obtained our desired bound on $i_1^{\prime\prime}$.\end{proof} 

\subsection*{Bessel functions of the second kind}
Each of the Bessel equations \eqref{besseleqn} and \eqref{modbesseleqn} is a second-order differential equation, and so has another set of solutions. However, these functions are singular at the origin. We proved in Chapter~\ref{spectrum} that the eigenfunctions are smooth; thus either these singular solutions do not appear in the eigenfunctions, or they appear in a linear combination such that the singular terms cancel. In Lemma~\ref{no2ndkind}, we will prove that in fact there is no nontrivial linear combination that meets the smoothness condition.

Ultraspherical Bessel functions of the second kind solve \eqref{besseleqn} and are defined by
\begin{align*}
n_{l}(z) &= z^{-s}N_{s+l}(z)
\end{align*}
with $s=(d-2)/2$. Here $N_\nu(z)$ denotes a Bessel function of the second kind of order $\nu$.  Each $N_\nu(z)$ is linearly independent of $J_\nu(z)$ (see, for example, \cite[p. 358]{AShandbook}), so $n_l(z)$ is linearly independent of $j_l(z)$. The functions $N_l(z)$ are often written as $Y_l(z)$; we use $N_l$ to avoid confusion with the spherical harmonics $\Yl(\thetahat)$.

Ultraspherical modified Bessel functions of the second kind solve \eqref{modbesseleqn} and are defined by
\begin{align*}
k_{l}(z) &= z^{-s}K_{s+l}(z)\\
\text{with}\quad s &= \frac{d-2}{2},
\end{align*}
where $K_\nu(z)$ denotes a modified Bessel function of the second kind of order $\nu$. As before, the $k_l(z)$ are linearly independent of the $i_l(z)$.

We will need several properties of these functions.  For orders $\nu$ that are nonnegative integers, we have the following ascending series: (see, for example, \cite[p. 360, 375]{AShandbook})
\begin{align}
N_\nu(z)&=-\frac{1}{\pi}\sum_{k=0}^{\nu-1}\frac{(\nu-k-1)!}{k!}(z/2)^{2k-\nu}\nonumber\\
&\qquad+\frac{2}{\pi}\ln(z/2)J_\nu(z)-\frac{1}{\pi}\sum_{k=0}^\infty C_\nu(k) (-1)^kz^{2k+\nu}, \label{neven}\\
K_\nu(z)&=\frac{1}{2}\sum_{k=0}^{\nu-1}\frac{(\nu-k-1)!}{k!}(-1)^k(z/2)^{2k-\nu}\nonumber\\
&\qquad+(-1)^{\nu+1}\ln(z/2)I_\nu(z)+(-1)^\nu\frac{1}{2}\sum_{k=0}^\infty  C_\nu(k)(z/2)^{2k+\nu}, \label{keven}
\end{align}
where the coefficients $C_\nu(k)$ are nonzero real number depending on $\nu$ and $k$ as follows:
\[
 C_\nu(k) =\frac{\psi(k+1)+\psi(\nu+k+1)}{k!(\nu+k)!}.
\]
Here $\psi(n)$ is the digamma function with $\psi(1)\approx-.577$ and $\psi(n+1)=\psi(n)+1/n$ for $n\geq 1$. 

For positive noninteger orders $\nu$, we have the following relations between Bessel functions of the first kind and second kind (see, for example, \cite[p. 358, 375]{AShandbook}):
\begin{align}
N_\nu(z)&=\frac{J_\nu(z)\cos(\nu\pi)-J_{-\nu}(z)}{\sin(\nu\pi)}\label{nodd}\\
K_\nu(z)&=\frac{\pi}{2}\frac{I_{-\nu}(z)-I_{\nu}(z)}{\sin(\nu\pi)}\label{kodd},
\end{align}
When $\nu$ is an integer, we have that $J_\nu$ and $J_{-\nu}$ are linearly dependant.

Note that for all dimensions $d\geq 2$ and all nonnegative integers $l$, the functions $n_l(z)$ and $k_l(z)$ are continuous for $z>0$.

In Chapter~\ref{unitball}, we will find the exact solutions of our eigenvalue equation on the unit ball. We will need the following lemma in order to show that Bessel and modified Bessel functions of the second kind do not appear in the radial parts of the smooth solutions. 

\begin{lemma} \label{no2ndkind}
Let $a$, $b$ be positive constants, $a<b$. 

For $d\geq 2$ and all integers $l\geq 2$, there is no nontrivial linear combination 
\[
 R(z)=An_l(az)+Bk_l(bz)
\]
so that $R(z)$ is smooth at $z=0$.

For $d\geq 2$ and $l=1$, there is no nontrivial linear combination
\[
 R(z)=An_1(az)+Bk_1(bz)
\]
so that $R(z)$ is smooth at $z=0$ with $R(0)=0$.

For $d\geq 2$ and $l=0$, there is no nontrivial linear combination
\[
 R(z)=An_0(az)+Bk_0(bz)
\]
so that $R(z)$ is smooth at $z=0$ with $R'(0)=0$.
\end{lemma}

\begin{proof} The Bessel functions $n_l$ and $k_l$ are real-valued, so we may assume the constants $A$ and $B$ are real. Recall $s=(d-2)/2$. We treat the cases of dimension $d$ even and odd separately.

[Part 1.] Let $d\geq 3$ be odd, so that $s$ is an odd multiple of $1/2$. Then for any nonnegative integer $l$, we have $\cos((s+l)\pi)=0$ and $\sin((s+l)\pi)=\pm 1$. Thus by the definitions of $n_l$ and $k_l$, and the identities \eqref{nodd} and \eqref{kodd}, we have
\begin{align*}
R(z)&=An_l(az)+Bk_l(bz)\\
&=A(az)^{-s}N_{s+l}(az)+B(bz)^{-s}K_{s+l}(bz)\\
&=\pm\left(-A(az)^{-s}J_{-s-l}(az)+\frac{\pi}{2}B(bz)^{-s}\Big(I_{-s-l}(bz)-I_{s+l}(bz)\Big)\right)\\
&=\pm \sum_{k=0}^\infty\frac{(z/2)^{2k-2s-l}2^{-s}}{k!\Gamma(k+1-s-l)}
\left(\frac{\pi}{2}Bb^{2k-2s-l}-A(-1)^ka^{2k-2s-l}\right)\\
&\qquad\mp\frac{\pi}{2}Bi_l(bz),
\end{align*}
by the power series expansions \eqref{jseries} and \eqref{iseries} for $J_\nu$ and $I_\nu$. The terms in that contribute to the singularity at the origin are: 
\[
 \sum_{k=0}^{(d-2+l)/2}\frac{(z/2)^{2k-2s-l}}{k!\Gamma(k+1-s-l)}
\left(\frac{\pi}{2}Bb^{2k-2s-l}-A(-1)^ka^{2k-2s-l}\right).
\]
So long as $2-2s-l=4-d-l<0$, the function $R(z)$ has at least two terms $z^{-d+2-l}$ and $z^{4-d-l}$ that are singular at the origin (corresponding to $k=0$ and $k=1$ in the above sum).  Choosing $A$ and $B$ so that the singular terms cancel, we must have
\begin{align}
 A&=\frac{\pi}{2}B\left(\frac{a}{b}\right)^{d-2+l} &\text{for $k=0$}\label{part1a}\\
 A&=-\frac{\pi}{2}B\left(\frac{a}{b}\right)^{d-4+l} &\text{for $k=1$}\label{part1b}.
\end{align}
The constants $a$ and $b$ are positive; thus we must have $A=B=0$ for all $d$, $l$ with $4-d-1<0$; this inequality holds for all odd $d\geq 3$ and nonnegative integers $l$ except for the case of $d=3$ with $l=0$ or $l=1$.

When $d=3$ and $l=0$ or $l=1$, we only have one singular term (corresponding to $k=0$). Assume $B\neq0$, since if $B=0$ then $A=0$ by~\eqref{part1a} and~\eqref{part1b} and so the linear combination is trivial. Without loss of generality we may take $B=1$. Then the linear combination
\[
R(z)=\frac{\pi}{2}\left(\frac{a}{b}\right)^{l+1}n_l(az)+k_l(bz)
\]
is continuous at $z=0$. Differentiating, we see
\begin{align*}
 R'(z)&=\pm\sum_{k=1}^\infty\frac{(2k-1-l)(z/2)^{2k-2-l}}{2k!\Gamma(k+1/2-l)}
\left(\frac{\pi}{2}b^{2k-1-l}-\frac{\pi}{2}\left(\frac{a}{b}\right)^{l+1}(-1)^ka^{2k-1-l}\right)\\
&\qquad+\mp\frac{\pi}{2}bi_l'(bz).
\end{align*}

When $l=0$, we have
\begin{align*}
 R'(0)&=\frac{1}{2\Gamma(3/2)}\left(\frac{\pi}{2}b+\frac{\pi}{2}\left(\frac{a}{b}\right)a\right)-\frac{\pi}{2}bi_l'(0)\\
&=\frac{\pi}{4\Gamma(3/2)}\left(b-\frac{a^2}{b}\right)\neq 0,
\end{align*}
noting $i_0'(0)=i_1(0)=0$ by \eqref{i2} and the series expansion \eqref{iseries} for $i_1(z)$. Because $a\neq b$, we have $R'(0)\neq 0$ for $l=0$. Thus in order to have $R(z)$ smooth at $z=0$ with $R'(0)=0$, we must take both $A$, $B$ to be zero.

When $l=1$, we have
\begin{align*}
R(z)&= -\sum_{k=1}^\infty\frac{(z/2)^{2k-2}}{k!\Gamma(k-1/2)}
\left(\frac{\pi}{2}b^{2k-2}-\frac{\pi}{2}\left(\frac{a}{b}\right)^{2}(-1)^ka^{2k-2}\right)+\frac{\pi}{2}Bi_1(bz),
\end{align*}
Then since $i_1(0)=0$, we have
\[
 R(0)=-\frac{\pi}{2\Gamma(1/2)}\left(1+\frac{a^2}{b^2}\right)
\]
which is nonzero. Thus in order to have $R(z)$ smooth at $z=0$ with $R(0)=0$, we must take both $A$, $B$ to be zero.

[Part 2.] Ascending power series centered about zero with powers increasing by steps of $2$ and with lowest-order term $Cz^k$ will be represented by $\BigO(z^k)$. Thus $\frac{d}{dz}\BigO(z^k)=\BigO(z^{k-1})$ for $k\geq 1$ and $\frac{d}{dz}\BigO(1)=\BigO(z)$.

 Let $d\geq2$ be even; then $s=(d-2)/2$ is an integer. Thus by the definitions of $n_l$ and $k_l$, and the identities \eqref{neven} and \eqref{keven}, we have
\begin{align}
R(z)&=An_l(az)+Bk_l(bz)\nonumber\\
 &=2^{-s}\sum_{k=0}^{s+l-1}\frac{(s+l-k-1)!}{k!}\left(\frac{z}{2}\right)^{2k-2s-l}\left(\frac{-A}{\pi}a^{2k-2s-l}+\frac{B}{2}(-1)^kb^{2k-2s-l}\right)\nonumber\\ &\qquad+\ln(az/2)\frac{2A}{\pi}j_{l}(az)+(-1)^{s+l}\ln(bz/2)Bi_{l}(bz)+\BigO(z^l)\label{part2R}
\end{align}
The singular contributions at $z=0$ come from the terms involving $\ln(z)$ when $l=0$ and when $l\geq 0$ from the sum
\begin{equation}\label{singsum}
 \sum_{k=0}^{s+l-1}\frac{(s+l-k-1)!}{k!}\left(\frac{z}{2}\right)^{2k-2s-l}\left(\frac{-A}{\pi}a^{2k-2s-l}+\frac{B}{2}(-1)^kb^{2k-2s-l}\right),
\end{equation}
provided that at least one of $s$, $l$ is nonzero.
So long as $2-2s-l=4-d-l<0$ and $s$ and $l$ not both zero, the above sum contains at least two terms that are singular at the origin, corresponding to $k=0$ and $k=1$. Choosing $A$ and $B$ so that both of these singular terms cancel, we again find
\begin{align*}
 A&=\frac{\pi}{2}B\left(\frac{a}{b}\right)^{d-2+l} &\text{for $k=0$},\\
 A&=-\frac{\pi}{2}B\left(\frac{a}{b}\right)^{d-4+l} &\text{for $k=1$}.
\end{align*}
The constants $a$ and $b$ are positive; thus we must have $A=B=0$. We have $4-d-l<0$ for all even $d\geq 6$, for $d=4$ when $l\geq 1$, and for $d=2$ when $l\geq 3$. The remaining cases are $d=4$ with $l=0$ and $d=2$ with $l=0$, $1$, and $2$.

We address $d=4$ with $l=0$ first. The sum \eqref{singsum} contains the single term corresponding to $k=0$; thus the singular terms are
\[
\left(\frac{z}{2}\right)^{-2}\left(\frac{-A}{\pi}a^{-2}+\frac{B}{2}b^{-2}\right)
\]
and the logarithmic terms
\[
 \ln(az/2)\frac{2A}{\pi}j_{0}(az)-\ln(bz/2)Bi_{0}(bz) =\ln(z)\left(\frac{2A}{\pi}-B+\BigO(z^2)\right)+\BigO(1),
\]
since $j_0(0)=i_0(0)=1$.
To make $R(z)$ continuous at $r=0$, we must then take
\[
 A = \frac{\pi}{2}\frac{a^2}{b^2}B \qquad\text{and}\qquad  A = \frac{\pi}{2}B.
\]
Since $a<b$, this is only possible when both $A$ and $B$ are zero.

When $d=2$, we have $s=0$ and $n_l(z)=N_l(z)$ and $k_l(z)=K_l(z)$.  We first consider $l=1$ and $2$. In these cases, we only have one singular term, corresponding to $k=0$ in the sum \eqref{singsum}. Assume $B\neq0$; as before we may take $B=1$. Then the linear combination
\[
R(z)=\frac{\pi}{2}\left(\frac{a}{b}\right)^{l}N_l(az)+K_l(bz)
\]
is continuous at $z=0$, and so we have by \eqref{part2R} that
\[
R(z)=\ln(z)\left(\left(\frac{a}{b}\right)^{l}J_{l}(az)+(-1)^{l}\ln(z)I_{l}(bz)\right)+\BigO(z^l).
\]

If $l=1$, we have
\begin{align*}
 R'(z)&=\ln(z)\left(\left(\frac{a}{b}\right)aJ_{1}'(az)-bI_{l}'(bz)\right)+O(1)\\
 &= \frac{1}{2}\ln(z)\left(\frac{a^2}{b}-b+\BigO(z^2)\right)+\BigO(1).
\end{align*}
Since $a<b$, we have $|R'(z)|\to\infty$ as $z\to0$. Thus we must take $B=0$ (and hence $A=0$ in order for $R(z)$ to be continuous with $R'(z)$ continuous at $z=0$.

If $l=2$, we have
\begin{align*}
R(z)&=\ln(z)\left(\left(\frac{a}{b}\right)^{2}J_{2}(az)+I_{2}(bz)\right)+\BigO(z^2),
\end{align*}
and so
\begin{align*}
R'(z)&=\ln(z)\left(\left(\frac{a}{b}\right)^{2}aJ_{2}'(az)+bI_{2}'(bz)\right)+\BigO(z),
\end{align*}
and
\begin{align*}
R''(z)&=\ln(z)\left(\left(\frac{a}{b}\right)^{2}a^2J_{2}^{\prime\prime}(az)+b^2I_{2}^{\prime\prime}(bz)\right)+\BigO(1).
\end{align*}
Because $J_{2}^{\prime\prime}(0)=I_{2}^{\prime\prime}(0)=1/4$, we have that 
as $z\to0$, $R''(z)$ behaves like
\[
 \ln(z)\frac{1}{4}\left(\left(\frac{a}{b}\right)^{2}a^2+b^2\right)+\BigO(1),
\]
and so we find $|R''(z)|\to\infty$ as $z\to0$. Thus we must take $A$ and $B$ both zero in order for $R(z)$ to be continuous with $\Rpp(z)$ continuous.

When $l=0$, the sum \eqref{singsum} is in fact empty, and the logarithmic terms equal
\[
 \ln{z}\left(\frac{2A}{\pi}J_{0}(az)+BI_{0}(bz)\right)+\BigO(1)
\]
Thus $R(z)=AN_0(az)+BK_0(bz)$ is continuous at $z=0$ if
\[
 A=-\frac{\pi}{2}B.
\]
As before, we assume $B=1$. Then
\begin{align*}
 R(z)&=-\ln(z)\left(J_0(az)-I_0(bz)\right)+\BigO(1).
\end{align*}
Thus the first derivative is
\begin{align*}
 R'(z) &=-\ln(z)\left(aJ_0'(az)-bI_0'(bz)\right)+\BigO(z)
\end{align*}
and the second derivative is given by
\begin{align*}
 \Rpp(z) &=-\ln(z)\left(a^2J_0''(az)-b^2I_0''(bz)\right)+\BigO(1)
\end{align*}
Then as $z\to0$, we see $|\Rpp(z)|\to\infty$. Thus we must take $A$ and $B$ both zero in order for $R(z)$ to be continuous with $\Rpp(z)$ continuous.
\end{proof}

%% file: unitball.tex
\chapter{The unit ball}\label{unitball}
We will use the eigenfunctions of the ball as our trial functions in the proof of the isoperimetric inequality, Theorem~\ref{thm1}. Happily, the full set of solutions for the ball can be found exactly in terms of Bessel and modified Bessel functions, and we can in fact identify the fundamental mode. In particular, the fundamental mode will be proved to have angular dependence.

We will focus on the unit ball, since the solution of our eigenvalue problem for any ball can then be obtained by scaling. We will show in Theorem~\ref{thm2} that all eigenfunctions will be of the form $R_l(r)\Yl(\thetahat)$, where $R_l$ is a linear combination (depending on $\tau$) of ultraspherical Bessel and modified Bessel functions of order $l$, and $\Yl$ is a spherical harmonic. 

\section*{Spherical harmonics}
In the case where $\Omega$ is the ball, it is natural to consider spherical coordinates. Let $r$ be the radius and $\thetahat$ be the remaining angular information. Consider Laplace's equation $\Delta f = 0$, with $f$ a function on $\RR^d$. The Laplacian can be written in spherical coordinates as
\[
\Delta = \partial_{rr}+\frac{d-1}{r}\partial_r+\frac{1}{r^2}\Delta_S,
\]
where we give the name $\Delta_S$ to the angular part of the Laplacian. Separating variables so that $f=R(r)Y(\thetahat)$, we obtain
\[
 R''+\frac{d-1}{r}R'-\frac{l(l+d-2)}{r^2}R=0 \qquad\text{and}\qquad \Delta_S Y=-l(l+d-2)Y.
\]
Using our earlier notation for the surface Laplacian, we have $\Delta_{\partial\Omega}=\frac{1}{R^2}\Delta_S$, when $\Omega$ is the ball of radius $R$. The parameter $l$ appearing in the separation constant $l(l+d-2)$ must be an nonnegative integer in order for solutions to exist. The solutions to $\Delta_S Y = -l(l+d-2)Y$ are called the \emph{spherical harmonics}. For each $l$, we choose a spanning set $\{\Yl\}$ of such solutions that are orthonormal with respect to the $L^2(\partial\BB)$ norm. Because the eigenvalues are real, the $\Yl$ may be chosen to be real-valued. However, they are traditionally chosen to be complex-valued, and so will be treated as possibly such in the proof of this chapter's main result, Theorem~\ref{thm2}.

\section*{Factoring the eigenfunction equation}
\begin{prop} Let $\tau>0$ and $\omega$ be any positive eigenvalue of the free plate when $\Omega$ is the unit ball. Then the corresponding eigenfunctions in Theorem~\ref{thm1} can be written in the form $R(r)\Yl(\thetahat)$, where $\Yl$ is a spherical harmonic of some integer order $l$ and
\[
R(r) = j_l(ar)+\gamma i_l(br),
\]
where $a$ and $b$ are positive constants depending on $\tau$ and $\omega$ as follows: $b^2-a^2=\tau$ and $a^2b^2=\omega$, and $\gamma$ is a real constant given by
\[
 \gamma=\frac{-a^2j_l^{\prime\prime}(a)}{b^2i_l^{\prime\prime}(b)}.
\]
\end{prop}

\begin{proof}
We first show that eigenfunctions can be written as a product of a radial function with a spherical harmonic, and then give the exact form of the radial part.

Write $A:=\Delta^2-\tau\Delta$. By Proposition~\ref{spect}, each eigenvalue $\omega$ has finite multiplicity, and so the corresponding space of eigenfunctions $X_\omega$ is finite-dimensional. Because $\Delta_S$ is independent of $r$, it commutes with the Laplacian $\Delta$ and hence with our operator $A:=\Delta^2-\tau\Delta$. Thus $\Delta_S$ maps $X_\omega$ into itself. The operator $\Ls$ is symmetric, and so diagonalizable on the finite-dimensional space $X_\omega$. The eigenfunctions of $\Ls$ on $\partial\BB$ are the spherical harmonics; on $\BB$ the eigenfunctions have the form $R(r)\Yl(\thetahat)$. Thus we can choose our eigenfunctions of $A$ to have this form. That is, $A$ and $\Ls$ are simultaneously diagonalizable.

To find the precise form of $R$, we factor the eigenvalue equation \eqref{maineq}, obtaining
\begin{equation}
(\Delta+a^2)(\Delta-b^2)u = 0, \label{PDEfactored}
\end{equation}
where $a$ and $b$ are positive real numbers satisfying $b^2=a^2+\tau$ and $\omega = a^2(a^2+\tau)$. That is, $a^2=\sqrt{(\tau/2)^2+\omega}-\tau/2$ and $b^2=\sqrt{(\tau/2)^2+\omega}+\tau/2$. The eigenfunctions $u$ will then be linear combinations of the solutions $v$ and $w$ of each factor:
\begin{equation}\label{factoredpieces}
 (\Delta+a^2)v=0 \qquad\text{and}\qquad(\Delta-b^2)w=0.
\end{equation}
Each of these is separable in spherical coordinates, with angular equation \\$\Delta_S Y=-l(l+d-2)Y$ for some nonnegative integer $l$. The radial equation for $v$ is a rescaling of the ultraspherical Bessel equation \eqref{besseleqn} with order $l$ and the radial equation for $w$ is a rescaling of the ultraspherical modified Bessel equation \eqref{modbesseleqn} with order $l$, hence
\[
 v=\Big(Aj_{l_1}(ar)+Bn_{l_1}(ar)\Big)Y_{l_1}\qquad\text{and}\qquad w=\Big(Ci_{l_2}(br)+Dk_{l_2}(br)\Big)Y_{l_2},
\]
for some nonnegative integers $l_1$, $l_2$ and real constants $A$, $B$, $C$, and $D$.
From the diagonalization argument above, we know $u=R(r)\Yl(\thetahat)$, so all the orders must agree: $l=l_1=l_2$. Thus solutions of the eigenvalue equation \eqref{maineq} have the form
\[
u(x)=R(r)T(\hat{\theta}) = \Big(Aj_l(ar)+Bn_l(ar)+Ci_l(br)+Dk_{l_2}(br)\Big)Y_l(\hat{\theta}).
\]
However, we have from Proposition~\ref{regprop} that the eigenfunctions are smooth on $\BB$. The spherical harmonics $\Yl$ have no radial dependence; thus we must have the radial part $R(r)$ be smooth for $r\in[0,\infty)$. When $l=0$, the spherical harmonic $Y_0$ is constant, and we must also require $R'(0)=0$ in order for $u$ to be smooth. When $l=1$, the spherical harmonics $Y_1$ can be given by $x_i/r$, where $x_i$ are the coordinate functions. Then along the $x_i$-axis, $R(r)Y_1(\thetahat)=R(r)x_i/r=R(r)\sign(x_i)$. This function is continuous at the origin only if $R(r)$ vanishes at $r=0$. By Lemma~\ref{no2ndkind}, there is no nontrivial linear combinination of Bessel functions of the second kind which satisfies these conditions; thus $B$ and $D$ are both zero. Denote $C/A$ by the constant $\gamma$; then we have
\[
u(x)=R(r)T(\hat{\theta}) = \Big(j_l(ar)+\gamma i_l(br)\Big)Y_l(\hat{\theta}).
\]
The constant $\gamma$ must be chosen so that $u$ satisfies the natural boundary condition $u_{rr}=0$ at $r=1$ (see \eqref{BCb1} in Proposition~\ref{ballBC}); hence we have
\[
 \gamma =\frac{-a^2j_l^{\prime\prime}(a)}{b^2i_l^{\prime\prime}(b)},
\]
and so $\gamma$ is real-valued.
\end{proof}

\section*{The fundamental mode of the ball}
In this section, we identify the fundamental mode of the ball for positive tension, proving:

\begin{thm} \label{thm2} For $\tau>0$, the fundamental mode of the unit ball has angular dependence, and has the form
\[
u_1(r,\thetahat)=\Big(j_1(ar)+\gamma i_1(br)\Big)Y_1(\thetahat),
\]
with $a$, $b$, $\gamma$ real constants, with $a$ and $b$ positive and depending on $\tau$ and $\omega$ as follows: $b^2-a^2=\tau$ and $a^2b^2=\omega$, and $\gamma$ given by
\[
 \gamma=\frac{-a^2j_l^{\prime\prime}(a)}{b^2i_l^{\prime\prime}(b)}.
\]
Thus in dimension 2,
\[
 u_1(r,\thetahat)=\Big(J_1(ar)+\gamma I_1(br)\Big)\begin{cases}\sin(\theta)\\ \cos(\theta)\end{cases}.
\]
\end{thm}
\begin{proof}
The proof will have two parts. First we show that for any radial function $R(r)$, the Rayleigh quotient $Q[R\Yl]$ is minimized when $l=1$, among all $l\geq1$. Then we show that of all nonconstant eigenstates with $l=0$ and $l=1$, the lowest eigenvalue corresponds to $l=1$. Note that when $l=0$, the spherical harmonic $Y_0$ is the constant function, and so $l=0$ corresponds to purely radial modes.

\emph{[Part 1.]} We will show that for any fixed smooth radial function $R$, the Rayleigh quotient $Q[R\Yl]$ is an increasing function in $l$ for all $l\geq 1$. Then by the variational characterization of eigenvalues, we see that the lowest eigenvalue corresponding to an eigenfunction with angular dependence (i.e., $l\geq 1$) occurs when $l=1$.

Considering the numerator and denominator separately, we will use the $L^2$-orthonormality of the spherical harmonics to simplify the angular parts of the integrals.

The denominator of our Rayleigh quotient is, for $u=R\Yl$,
\[
\int_\BB |R\Yl|^2\,dx = \int_0^1R^2\,r^{d-1}\,dr,
\]
and so is independent of $l$. So it suffices to show that the numerator is an increasing function of $l$ for $l\geq 1$.

Recall the numerator of the Rayleigh Quotient is
\[
N[u] = \int_\Omega |D^2u|^2+\tau |Du|^2 \, dx.
\]

We use the pointwise identity of Fact~\ref{ptwise} to rewrite the Hessian term as :
\begin{align}\label{hessstuff}
|D^2u|^2 &= \frac{1}{2}\Big(\Delta(|Du|^2) -D(\Delta u)\cdot D\ubar-D(\Delta\ubar)\cdot Du\Big)
\end{align}
Because our region $\Omega$ is the unit ball, we may use spherical coordinates, noting $\frac{\partial u}{\partial n} = u_r$. Recall that
\begin{equation}
D u = u_r\rhat + \frac{1}{r}\Ns u,
\end{equation}
where $\rhat = x/r$ is the unit normal. Recall from the derivation of the boundary conditions on the ball from Chapter~\ref{spectrum} that $\frac{1}{r}\Ns=\sgrad$ is the surface gradient and $\frac{1}{r^2}\Ls=\sdiv\sgrad$ is the Laplacian on the boundary of the ball. 

Note by the Divergence Theorem on $\partial\BB$, we have for any function $f$,
\begin{equation}\label{laptozero}
  \surfintB \frac{1}{r^2}\Ls f\,dS=\surfintB \mydiv_{\partial\BB}\Big(\grad_{\BB}f\Big)\,dS=0.
\end{equation}

Further exploiting orthonormality of the $\Yl$, we see that
\begin{equation}\label{laprayleigh}
 \surfintB |\Ns\Yl|^2\,dS=l(l+d-2)=:k.
\end{equation}
Thus when $u=R\Yl$, we can rewrite \eqref{hessstuff} as follows:
\begin{align*}
\frac{1}{2}&\int_\BB\Big(\Delta(|Du|^2) -D(\Delta u)\cdot D\ubar-D(\Delta\ubar)\cdot Du\Big)\,dx\\
&=\frac{1}{2}\int_\BB\left(\frac{\partial^2}{\partial r^2}+\frac{d-1}{r}\frac{\partial}{\partial r}+\frac{1}{r^2}\Ls\right)\left((R')^2|\Yl|^2+\frac{R^2}{r^2}|\Ns\Yl|^2\right)\\
&\qquad-D\ubar\cdot D\left(u_{rr}+\frac{d-1}{r}u_r-\frac{k}{r^2}u\right)-Du\cdot D\left(\ubar_{rr}+\frac{d-1}{r}\ubar_r-\frac{k}{r^2}\ubar\right)\,dx\\
&\qquad\text{(noting $\Ls\Yl=-k\Yl$)}\\
&=\frac{1}{2}\int_\BB\left(\frac{\partial^2}{\partial r^2}+\frac{d-1}{r}\frac{\partial}{\partial r}\right)\left((R')^2|\Yl|^2+\frac{R^2}{r^2}|\Ns\Yl|^2\right)\\
&\qquad -D\ubar\cdot D\left(u_{rr}+\frac{d-1}{r}u_r-\frac{k}{r^2}u\right)-Du\cdot D\left(\ubar_{rr}+\frac{d-1}{r}\ubar_r-\frac{k}{r^2}\ubar\right)\,dx,
\end{align*}
with this last by noting that \eqref{laptozero} gives us
\[
 \int_\BB\frac{1}{r^2}\Ls\left((R')^2|\Yl|^2+\frac{R^2}{r^2}|\Ns\Yl|^2\right)\,dx=0.
\]
Expanding the integrands, we see \eqref{hessstuff} becomes
\begin{align*}
\int_\BB&|\Yl|^2\left((\Rpp)^2+\frac{k+d-1}{r^2}(\Rp)^2-\frac{2k}{r^3}R\Rp\right)\\
&\qquad+|\Ns\Yl|^2\left(\frac{(\Rp)^2}{r^2}-\frac{4}{r^3}R\Rp+\frac{k-d+4}{r^4}R^2\right)\,dS
\end{align*}
Then integrating the above over $\BB$ using \eqref{laprayleigh} and the orthonormality of the $\Yl$, we obtain
\begin{align*}
\int_\BB&|D^2u|^2\,dx\\ &=\int_0^1\left((\Rpp)^2+\frac{2k+d-1}{r^2}(\Rp)^2-\frac{6k}{r^3}R\Rp+\frac{k(k-d+4)}{r^4}R^2\right)r^{d-1}\,dr\\
&=\int_0^1\left((\Rpp)^2+\frac{d-1}{r^2}(\Rp)^2+\frac{2k}{r^4}\left(r\Rp-\frac{3}{2}R\right)^2+\frac{k(k-d-1/2)}{r^4}R^2\right)r^{d-1}\,dr
\end{align*}
with this last equality by completing the square.

We now examine the gradient term in the numerator of the Rayleigh quotient:
\begin{align*}
\int_\BB|Du|^2\,dx &= \int_\BB\left(|u_r|^2+\frac{1}{r^2}|\Ns u|^2\right)\,dx\nonumber\\
&=\int_\BB\left((\Rp)^2|\Yl|^2+\frac{R^2}{r^2}|\Ns\Yl|^2\right)\,dx\nonumber\\
&=\int_0^1\left((\Rp)^2+\frac{k}{r^2}R^2\right)r^{d-1}\,dr,
\end{align*}
again by \eqref{laprayleigh}.

Combining these results, the numerator of the Rayleigh quotient can be now written with all $k$-dependence (and hence $l$-dependence) explicit:
\begin{align*}
N[u]&=\int_0^1\left(\frac{2k}{r^4}\left(r\Rp-\frac{3}{2}R\right)^2+\frac{k(k-d-1/2)}{r^4}R^2+\tau\frac{kR^2}{r^2}\right)r^{d-1}\,dr\\
&\qquad+\int_0^1\left((\Rpp)^2+\frac{d-1}{r^2}(\Rp)^2+\tau(\Rp)^2\right)r^{d-1}\,dr
\end{align*}
Recall $\tau>0$; then the above is increasing with $k$ for $k\geq d+1/2$. Recall $k=l(l+d-2)$ is increasing as a function of $l$; then all terms involving $k$ are increasing functions of $l$ for $l\geq 2$. If $l=1$, the expression $k(k-d-1/2)$ becomes $-3(d-1)/2$, which is negative for all dimensions under consideration. If $l=2$, we find $k(k-d-1/2)= 2d(d-1/2)>0$. Thus each term involving $l$ is increasing as a function of $l$ for all $l\geq 1$. 

Thus for any fixed radial function $R$, the numerator $N[R\Yl]$ is an increasing function of $l$ for $l\geq 1$.

\emph{[Part 2.]} We now show that the lowest eigenvalue corresponding to an eigenfunction of the form $u_l=\Big(j_l(ar)+\gamma i_l(br)\Big)\Yl$ with $l=1$ is less than the lowest positive eigenvalue for $l=0$.  

Recall $b = \sqrt{a^2+\tau}$. Let $\ainf$ denote the first positive zero of $j_1'(a)$. All eigenfunctions $u$ satisfy the natural boundary conditions, $Mu=0$ and $Vu=0$ on $\partial\Omega$, as given in \eqref{BCb1} and \eqref{BCb2} for the ball. Since all eigenfunctions are linear combinations of $j_l(ar)\Yl(\thetahat)$ and $i_l(br)\Yl(\thetahat)$, we must have some nontrivial linear combination satisfy the homogeneous linear equations
\[
 \begin{cases} Mu=0\\Vu=0.
  \end{cases}
\]
Thus we need the determinant
\[
W_l(a):=\det\begin{pmatrix}
Mj_l(ar) & Mi_l(br) \\
Vj_l(ar) & Vi_l(br)\end{pmatrix}_{r=1}
\]
to vanish. Using the form of natural boundary conditions for the ball given in Proposition~\ref{ballBC}, we have
\begin{equation} \label{Ms}
 Mj_l(ar)=a^2j_l^{\prime\prime}(ar) \qquad\text{and}\qquad Mi_l(ar)=b^2i_l^{\prime\prime}(br).
\end{equation}
The $j_l(ar)$ and $i_l(br)$ are rescaled ultraspherical Bessel and modified Bessel functions, so by the factorization \eqref{factoredpieces}, we have
\[
 \Delta j_l(ar)\Yl(\thetahat)=-a^2j_l(ar)\Yl(\thetahat) \qquad\text{and}\qquad \Delta i_l(br)\Yl(\thetahat)=b^2i_l(br)\Yl(\thetahat).
\]
Then noting $r=1$ on $\partial\BB$, the ``$V$`` boundary condition terms from in \eqref{BCb2} can be rewritten as follows:
\begin{align*}
 Vj_l(a)&= \tau aj_l^\prime(a)+k\Big(aj_l^\prime(a)-j_l(a)\Big)+a^3j_l'(a)\\
 Vi_l(b)&= \tau bi_l^\prime(b)+k\Big(bi_l^\prime(b)-i_l(b)\Big)-b^3i_l'(b)\\
\end{align*}

Combining the above with \eqref{Ms} and substituting $\tau=b^2-a^2$, we find
\begin{align}
W_l(a) &= a^2j_l^{\prime\prime}(a)\Big(-a^2bi_l^{\prime}(b)+k(bi_l^\prime(b)-i_l(b))\Big)\nonumber\\
&\quad-b^2i_l^{\prime\prime}(b)\Big(ab^2j_l^\prime(a)+k(aj_l^\prime(a)-j_l(a))\Big)\label{MVl}.
\end{align}

Given $\tau$, the roots of $W_l(a)$ will determine the eigenvalues by the relation $\omega=a^2(a^2+\tau)$. The parameter $\tau$ is positive, so $\omega$ increases with $a$. Note that if $a=0$ is a root, it corresponds to $\omega=0$.

 Therefore to show that the lowest nonzero eigenvalue corresponds to $l=1$ and not $l=0$, we show that the first nonzero root of $W_1(a)$ is less than the first nonzero root of $W_0(a)$.

First we consider $l=0$. Here $k = 0$, so we look for solutions to:
\begin{align*}
0&=W_0(a)\\ &=a^2j_0^{\prime\prime}(a)\Big(-a^2bi_0^{\prime}(b)\Big)-b^2i_0^{\prime\prime}(b)\Big(ab^2j_0^\prime(a)\Big)\\
&=a^4bj_1^\prime(a)i_1(b)+ab^4i_1^\prime(b)j_1(a),
\end{align*}
by \eqref{j2} and \eqref{i2}. The functions $i_0(b)$ and $i_0^\prime(b)$ are positive for $b>0$, as noted in Chapter~\ref{bessel} by the power series expansion \eqref{iseries}. Similarly, $j_1(a)$ and $j_1'(a)$ are positive on $(0,\ainf)$ by Lemma~\ref{fact1.5}, so $W_0(a) > 0$ on $(0,\ainf)$.

Now consider $l=1$. The constant $k=d-1$, so we have
\begin{align}
W_1(a)&=a^2j_1^{\prime\prime}(a)\Big(-a^2bi_1^\prime(b)+(d-1)(bi_1^\prime(b)-i_1(b))\Big)\nonumber\\
&\quad-b^2i_1^{\prime\prime}(b)\Big(ab^2j_1^\prime(a)+(d-1)(aj_1^\prime(a)-j_1(a))\Big)\label{W1a}\\
&=a^2bj_1^{\prime\prime}(a)\Big(-a^2i_1^\prime(b)+(d-1)i_2(b)\Big)\nonumber\\
&\quad-ab^2i_1^{\prime\prime}(b)\Big(b^2j_1^\prime(a)-(d-1)j_2(a)\Big),\label{W1b}
\end{align}
by the Bessel identities \eqref{j2} and \eqref{i2}. As $a\rightarrow 0$, the first term in \eqref{W1b} behaves like
\[
a^2\sqrt{\tau}j_1^{\prime\prime}(a)\Big((d-1)i_2(\sqrt{\tau})\Big)
\]
and so its sign is determined by that of $j_1^{\prime\prime}(a)$. By Lemma~\ref{fact3}, $j_1^{\prime\prime}(a)$ is negative for all $a\in(0,\ainf]$. Hence the first term of $W_1(a)$ is negative near $a=0$.

Similarly, as $a\rightarrow0$, we see $j_2(a)\rightarrow 0$, and so the second term behaves like
\[
-a\tau^2i_1^{\prime\prime}(\sqrt{\tau})j_1^\prime(0),
\]
which is negative by Lemma~\ref{fact1.5}. Therefore $W_1(a) < 0$ near $a=0$.

At $a = \ainf$, we have $j_1^\prime(\ainf) = 0$ by definition of $\ainf$; also note $j_1(\ainf)>0$ and $j_1^{\prime\prime}(\ainf)<0$ by Lemma~\ref{fact1} and~\ref{fact3}. Write $b_\infty=\sqrt{\ainfsq+\tau}$. By Proposition~\ref{propLS}, we have $\ainfsq>d>d-1$ for $d\geq 3$; for $d=2$, we have $\ainfsq \approx 1.84^2>d-1$. Thus $\ainfsq-(d-1)>0$, and so \eqref{W1a} gives
\begin{align*}
W_1(\ainf)&= -\ainfsq j_1^{\prime\prime}(\ainf)\Big(b_\infty(\ainfsq-(d-1))i_1^\prime(b_\infty) +(d-1)i_1(b_\infty)\Big)\\
&\qquad +b_\infty^2i_1^{\prime\prime}(b_\infty)(d-1)j_1(\ainf).
\end{align*}
Hence both terms in $W_1(\ainf)$ and since $W_1$ is continuous, it must have a zero in $(0,\ainf)$. Thus the lowest nonzero root of $W_l(a)$ occurs when $l=1$, not $l=0$, and so the lowest eigenvalue $\omega=a(a^2+\tau)$ occurs when $l=1$. \end{proof}

We have now assembled all of the tools we will need to tackle the proof of the main theorem.

%% file: maintheorem.tex
\chapter{Proof of the free plate isoperimetric inequality}\label{mainthm}
In this chapter, we prove Theorem~\ref{thm1}:

\emph{Among all regions of a fixed volume, when $\tau>0$ the fundamental tone of the free plate is maximal for $\Omega$ a ball. That is,}
\begin{equation}
\omega_1(\Omega) \leq \omega_1(\Ostar), \qquad\text{with equality if and only if $\Omega$ is a ball.} \label{c2ineq}
\end{equation}

For simplicity, in this section we will write $\omega$ instead of $\omega_1$ for the fundamental tone of the free plate with shape $\Omega$; the fundamental tone of the unit ball will be denoted by $\ostar$. When dependence on the region $\Omega$ and the tension $\tau$ need be made explicit, we write $\omega(\tau,\Omega)$ for the fundamental tone and $Q_{\tau,\Omega}$ for the Rayleigh quotient. We shall also need the notation $s\Omega := \{x\in\RR^d:x/s\in\Omega\}$ for $s>0$. Note also that $\tau>0$ throughout this section.

The proof of Theorem~\ref{thm1} will proceed from a series of lemmas, following roughly this outline:
\begin{itemize}
\item Scaling
\item Definition of trial functions
\item Concavity of the radial part of the trial function
\item Evaluation of the Rayleigh Quotient
\item Partial monotonicity of the integrand in the numerator and denominator
\item A collection of lemmas needed to establish this monotonicity
\item Rearrangement
\item Proof of the theorem.
\end{itemize}

We begin by rescaling to reduce to domains having the same volume as the unit ball $\BB$ (Lemma~\ref{scaling}). Adapting Weinberger's approach for the membrane \cite{W56}, we construct in Lemma~\ref{trialfcn} trial functions with radial part $\rho$ matching the radial part of the fundamental mode of the ball. We follow by proving in Lemma~\ref{gppneg} a concavity property of $\rho$ that will be needed later on. We next bound the eigenvalue $\omega$ by a quotient of integrals over our region $\Omega$, both of whose integrands are radial functions (Lemma~\ref{lemmaboundRC}). These integrands will be shown to have a "partial monotonicity". The denominator's integrand is increasing by Lemma~\ref{mondenom} and the numerator's integrand satisfies a decreasing partial monotonicity condition by Lemma~\ref{monnum}. The proof of Lemma~\ref{monnum} becomes rather involved and so is broken into two cases, Lemma~\ref{largetau} for large $\tau$ values, and Lemma~\ref{smalltau} for small values of $\tau$. The latter in turn requires some facts about particular polynomials, proved in Lemmas~\ref{poly1} and~\ref{poly2}. We then exploit partial monotonicity to see that the quotient of integrals is bounded above by the quotient of the same integrals taken over $\Ostar$, by Lemma~\ref{monint}. Finally, we conclude that the quotient of integrals on $\Ostar$ is in fact equal to the eigenvalue $\ostar$ of the unit ball. From there we deduce the theorem.

Our first lemma is a scaling argument.
\begin{lemma}\label{scaling} (Scaling) For all $s>0$, we have
\[
 \omega(\tau,\Omega)=s^{4}\omega(s^{-2}\tau,s\Omega).
\]
\end{lemma}
\begin{proof}
For any $u\in H^2(\Omega)$ with $\int_\Omega u\,dx=0$, let $\tilde{u}(x)=u(x/s)$. Then $\tilde{u}$ is a valid trial function on $s\Omega$ and so
\begin{align*}
Q_{s^{-2}\tau,s\Omega}[\tilde{u}]&=\frac{\int_{s\Omega}|D^2\tilde{u}|^2+s^{-2}\tau|D\tilde{u}|^2\,dx}{\int_{s\Omega}\tilde{u}^2\,dx}\\
&=\frac{\int_{s\Omega}|s^{-2}(D^2u)(x/s)|^2+s^{-2}\tau|s^{-1}(Du)(x/s)|^2\,dx}{\int_{s\Omega}u(x/s)^2\,dx}\\
&=\frac{s^{-4+d}\int_\Omega|D^2u|^2+\tau|Du|^2\,dy}{s^{d}\int_\Omega u^2\,dy} \qquad\qquad\qquad\text{taking $y=x/s$,}\\
&=s^{-4}Q_{\tau,\Omega}[u].
\end{align*}
Now the lemma follows from the variational characterization of the fundamental tone.
\end{proof}

Once we have established inequality \eqref{c2ineq} for all regions $\Omega$ of volume equal to that of the unit ball and all $\tau>0$, we obtain \eqref{c2ineq} for regions of arbitrary volume, since
\[
\omega(\tau,\Omega)=s^{4}\omega(s^{-2}\tau,s\Omega)\leq s^{4}\omega(s^{-2}\tau,s\Ostar)=\omega(\tau,\Ostar),
\]
for all $s>0$.

Next, inspired by Weinberger's proof for the membrane \cite{W56}, we choose appropriate trial functions from the fundamental modes of the unit ball. In the following lemmas, we take
\[
R(r)=j_1(ar)+\gamma i_1(br)
\]
to be the radial part of the fundamental mode of the unit ball. Recall $a$ and $b$ are positive constants determined by $\tau$ and the boundary conditions, as in the proof of Theorem~2 in Chapter~\ref{unitball}. The constant $\gamma$ is positive and determined by $a$, $\tau$, and the boundary conditions to be
\begin{equation}
\gamma :=\frac{-a^2j_1^{\prime\prime}(a)}{b^2i_1^{\prime\prime}(b)}>0.\label{gammadef}
\end{equation}
Recall also that $R(r)>0$ on $(0,1]$ and $R'(1)>0$.
\begin{lemma}\label{trialfcn} (Trial functions) Let the radial function $\rho$ be given by the function $R$, extended linearly. That is,
\[
 \rho(r)=\begin{cases} R(r) &\text{when $0\leq r\leq1$,}\\
       R(1)+(r-1)\Rp(1)&\text{when $r\geq1$.}\\
      \end{cases}
\]
After translating $\Omega$ suitably, the functions $u_k = x_k\rho(r)/r$, for $k=1,\dots,d$, are valid trial functions for the fundamental tone.
\end{lemma}
\begin{proof}
To be valid trial functions, the $u_k$ must be in $H^2(\Omega)$. Because $\Omega$ is bounded in $\RR^d$, the only possible issue is at the origin. The series expansions \eqref{jseries} and \eqref{iseries} of $j_1$ and $i_1$ respectively give us that $R(r)/r$ approaches a constant as $r\to 0$. thus, $u_k\in H^2(\Omega)$ as desired. They must also be perpendicular to the constant, that is,
\[
\int_\Omega \frac{\rho(r)x_k}{r}\,dx=0\qquad\text{for $k=1,\dots,d$.}
\]
We use the Brouwer Fixed Point Theorem to translate our region so that the above conditions are guaranteed; here again we follow Weinberger \cite{W56}. Write $x=(x_1,\dots,x_d)$ and consider the vector field
\[
 X(v)=\int_\Omega\frac{\rho(|x-v|)}{|x-v|}(x-v)\,dx.
\]
The vector field $X$ is continuous by construction. Along the boundary of the convex hull of $\Omega$, $X(v)$ is inward-pointing, because $\rho\geq0$ and the entire region $\Omega$ lies in a half-space to one side of $v$. Thus by the Brouwer Fixed Point Theorem, our vector field $X$ vanishes at some $v$ in the convex hull of $\Omega$. If we first translate $\Omega$ by $v$, then we have $X(0)=\int_\Omega \rho(r)x/r\, dx=0$. This gives us $\int_\Omega u_k\,dx=0$, as desired.
\end{proof}
We will need one further fact about our radial function $\rho$.
\begin{lemma}\label{gppneg} (Concavity)
The function $\rpp(r)\leq0$ for $r\in[0,1]$, with equality only at the endpoints.
\end{lemma}
\begin{proof}
First note that on $[0,1]$, the function $\rho\equiv R$. We see
\[
R^{\prime\prime}(r) = a^2j_1^{\prime\prime}(ar)+\gamma b^2i_1^{\prime\prime}(br),
\]
which is zero at $r=0$ because the individual Bessel derivatives vanish there, by the series expansions in the proof of Lemma~\ref{ijbounds}. At $r=1$, the function $R^{\prime\prime}$ vanishes because of the boundary condition $Mu = 0$.

The fourth derivative of $R$ is given by
\[
R^{\prime\prime\prime\prime}(r) = a^4j_1^{\prime\prime\prime\prime}(ar)+\gamma b^4i_1^{\prime\prime\prime\prime}(br).
\]

Because all derivatives of $i_1(z)$ are positive when $z\geq 0$, the second term above is positive on $(0,\infty)$. We have by Lemma~\ref{fact4} that $j_1^{\prime\prime\prime\prime}(z)$ is positive on $(0,\ainf]$. Thus $R^{\prime\prime\prime\prime}(r)>0$ on $(0,1]$, and so $R^{\prime\prime}(r)$ is a strictly convex function on $[0,1]$. Since $R^{\prime\prime}=0$ at $r=0$ and $r=1$, the function $R^{\prime\prime}$ must be negative on the interior of the interval $[0,1]$.
\end{proof}

We now bound our fundamental tone above by a quotient of integrals whose integrands are radial functions. Write
\[
 N[\rho]:=(\rpp)^2+\frac{3(d-1)}{r^4}(\rho-r\rp)^2+\tau(\rp)^2+\frac{\tau(d-1)}{r^2}\rho^2.
\]
\begin{lemma}\label{lemmaboundRC} (Evaluation)
 For any $\Omega$, translated as in Lemma~\ref{trialfcn}, we have
\begin{equation}
\omega \leq \frac{\int_\Omega N[\rho]\,dx}{\int_\Omega \rho^2\,dx}\label{boundRC}
\end{equation}
with equality if $\Omega=\Ostar$.
\end{lemma}

\begin{proof}
For $u_k$ defined as in Lemma~\ref{trialfcn}, we have
\[
\omega \leq Q[u_k] = \frac{\int_\Omega|D^2u_k|^2+\tau|Du_k|^2\,dx}{\int_\Omega|u_k|^2\,dx},
\]
from the Rayleigh-Ritz characterization. We have equality when $\Omega=\Ostar$ because the $u_k$ are the eigenfunctions for the ball associated to the fundamental tone by Chapter~\ref{unitball}.  Multiplying both sides by $\int_\Omega|u_k|^2\,dx$ and summing over all $k$, we obtain
\begin{equation}
\omega \int_\Omega \sum_{k=1}^d|u_k|^2\,dx
\leq \int_\Omega\sum_{k=1}^d|D^2u_k|^2+\tau\sum_{k=1}^d|Du_k|^2\,dx\label{multsum}
\end{equation}
again with equality if $\Omega=\Ostar$.
 From Appendix~\ref{calcfact}, Fact~\ref{derivs}, we see inequality \eqref{multsum} becomes
\[
\omega \int_\Omega \rho^2\,dx \leq \int_\Omega\left((\rpp)^2+\frac{3(d-1)}{r^4}(\rho-r\rp)^2+\tau(\rp)^2+\frac{\tau(d-1)}{r^2}\rho^2\right)\,dx,
\]
once more with equality if $\Omega$ is the ball $\Ostar$. Dividing both sides by $\int_\Omega \rho^2\,dx$, we obtain \eqref{boundRC}.
\end{proof}

We want to show the quotient in Lemma~\ref{lemmaboundRC} has a sort of monotonicity with respect to the region $\Omega$, and so we examine the integrands of the numerator and denominator separately.
\begin{lemma}\label{mondenom}(Monotonicity in the denominator)
 The function $\rho(r)^2$ is strictly increasing.
\end{lemma}
\begin{proof}
Differentiating, we see
\[
 \rp(r)=\begin{cases} j_1^\prime(ar)+\gamma i_1^\prime(br) &\text{when $0\leq r \leq 1$,}\\
         R^\prime(1) &\text{when $r\geq1$.}
        \end{cases}
\]
Obviously $i_1^\prime(br)\geq 0$. Because we have $a<\ainf$ from Chapter~\ref{unitball}, the function $j_1^\prime(ar)$ is positive on $[0,1]$. Thus $\rp(r)$ is positive everywhere, and $\rho$ (and therefore $\rho^2$) is an increasing function.
\end{proof}

The monotonicity result for the numerator is rather more complicated and requires several lemmas. We do not need to prove the integrand of the numerator is strictly decreasing; a weaker "partial monotonicity" condition is sufficient. We will say a function $F$ is \emph{partially monotonic for $\Omega$} if it satisfies
\begin{equation}
F(x)> F(y) \qquad\text{for all $x\in\Omega$ and $y\not\in\Omega$.}\label{moncond}
\end{equation}

\begin{lemma}\label{monnum}(Partial monotonicity in the numerator)
The function
\[
N[\rho]=(\rpp)^2+\frac{3(d-1)}{r^4}(\rho-r\rp)^2+\tau\left((\rp)^2+(d-1)\frac{\rho^2}{r^2}\right)
\]
satisfies condition \eqref{moncond} for $\Omega$ the unit ball.
\end{lemma}
\begin{proof}
Given that $\rpp < 0$ on $(0,1)$ and equals zero elsewhere by Lemma~\ref{gppneg}, the function $(\rpp)^2$ satisfies condition \eqref{moncond} for the unit ball. The derivative of the function $\tau(\rp)^2$ with respect to $r$ is $2\tau\rp\rpp$, and hence negative on $(0,r)$ and zero everywhere else. Thus $\tau(\rp)^2$ is a decreasing function of $r$. It remains to show that the remaining term
\[
h(r) =\frac{3(\rho-r\rp)^2}{r^4}+\tau\frac{\rho^2}{r^2}
\]
is also a decreasing function of $r$. Differentiating, we see
\begin{align*}
h'(r) &=\frac{-2}{r^3}(\rho-r\rp)\left(\frac{6}{r^2}(\rho-r\rp)+3\rpp+\tau \rho\right).
\end{align*}
Now, $\rho-r\rp=0$ at $r=0$ and
\[
 \frac{d}{dr} (\rho-r\rp) = -r\rpp,
\]
so by Lemma~\ref{gppneg}, $(\rho-r\rp)$ is positive on $(0,\infty)$ and vanishes at zero. Thus in order for $h(r)$ to be decreasing, we must have
\begin{equation}
  \frac{6}{r^2}(\rho-r\rp)+3\rpp+\tau \rho>0. \label{quantinterest}
\end{equation}

Let $\Delta_r\rho := \rpp-(d-1)r^{-2}(\rho-r\rp)$. Recall from the Bessel equations~\eqref{besseleqn} and~\eqref{modbesseleqn} that
\begin{equation}
\Delta_rj_1(ar)=-a^2j_1(ar)\quad\text{and}\quad\Delta_ri_1(br)=b^2i_1(br).\label{delta}
\end{equation}
Then on the interval $[0,1]$,
\begin{align*}
\frac{6}{r^2}(\rho-r\rp)+3\rpp+\tau\rho &=\frac{6}{r^2}(\rho-r\rp)+3\left(\Delta_r\rho+\frac{d-1}{r^2}(\rho-r\rp)\right) +\tau\rho\\
&=\frac{3(d+1)}{r^2}(\rho-r\rp)+3\Big(-a^2j_1(ar)+\gamma b^2i_1(br)\Big)+\tau\rho,
\end{align*}
with the last equality by \eqref{delta}. Considering the first term of the last line above, we see by \eqref{j2} and \eqref{i2},
\begin{align*}
\frac{1}{r^2}(\rho-r\rp)&=\frac{1}{r^2}\Big(arj_2(ar)-br\gamma\, i_2(br)\Big)\\
&=\frac{1}{d+2}\left(a^2\Big(j_1(ar)+j_3(ar)\Big)+\gamma b^2\Big(i_3(br)-i_1(br)\Big)\right)
\end{align*}
with the second equality by \eqref{j1}, \eqref{i1}, and simplifying.

Therefore our quantity of interest in \eqref{quantinterest} can be bounded below in terms of $j_l$'s and $i_l$'s:
\begin{align*}
\frac{6}{r^2}(\rho-r\rp)&+3\rpp+\tau\rho \\
&=\left(\tau-\frac{3a^2}{d+2}\right)j_1(ar)+\frac{3a^2(d+1)}{d+2}j_3(ar)\\
&\quad+\gamma\left(\tau+\frac{3b^2}{d+2}\right)i_1(br)+\gamma\frac{3b^2(d+1)}{d+2}i_3(br).\\
&\geq \frac{3a^2(d+1)}{d+2}j_3(ar)+\gamma\frac{3b^2(d+1)}{d+2}i_3(br)\\
&\quad+\left(\tau-\frac{3a^2}{d+2}+\gamma\left(\tau+\frac{3b^2}{d+2}\right)\right)j_1(ar)
\end{align*}
with the inequality by $j_l(ar) \leq i_l(ar)\leq i_l(br)$, since $\tau>0$ and so $a<b$.

The function $i_3$ is everywhere positive. We have $a<\ainf$, so $ar<\ainf$ on $[0,1]$ and hence the functions $j_3(ar)$ and $j_1(ar)$ are positive on $[0,1]$ by Lemma~\ref{fact1}. The remaining factor is positive for all $\tau>0$ by Lemmas~\ref{largetau} and \ref{smalltau} (to follow), thus establishing \eqref{quantinterest} and completing the proof.
\end{proof}

We establish the positivity of the remaining factor first for those values of\\ $\tau>9/(d+5)$; the proof for the remaining $\tau$ values is more complicated and is treated in another lemma.
\begin{lemma}\label{largetau} (Large $\tau$) We have
\begin{equation}
\tau-\frac{3a^2}{d+2} > 0  \label{tclem}
\end{equation}
for all $\tau>9/(d+5)$.
\end{lemma}
\begin{proof} We use the bounds we established for $\omega(\tau)$ in Chapter~\ref{spectrum}.

Recall that the first free membrane eigenvalue for the ball is $\mu_1^*=\ainfsq$. Lemma~\ref{wbounds} and Proposition~\ref{propLS} together give $(d+2) \tau > \ostar > \tau d$. Because $\ostar = a^4+a^2\tau$ from our factoring of the eigenvalue equation in Chapter~\ref{unitball}, we obtain inequalities relating $\tau$ and $a$:
\begin{equation}
\frac{a^4}{d-a^2}> \tau > \frac{a^4}{d+2-a^2}, \label{abounds}
\end{equation}
with the upper bound holding only if $a^2<d$.

Using the lower bound, we see
\begin{align*}
\tau-\frac{3a^2}{d+2}&>\frac{a^4}{d+2-a^2}-\frac{3a^2}{d+2}\\
&= \frac{a^4(d+5)-3a^2(d+2)}{(d+2)(d+2-a^2)}
\end{align*}
which is nonnegative whenever $a^2 \geq 3(d+2)/(d+5)$. When $a^2 <3(d+2)/(d+5)$, we have
\[
\tau - \frac{3a^2}{d+2} >\tau- \frac{9}{d+5}>0,
\]
by our choice of $\tau$.
\end{proof}

\begin{lemma}\label{smalltau} (Small $\tau$) We have
\begin{equation}
 \tau-\frac{3a^2}{d+2}+\gamma\left(\tau+\frac{3b^2}{d+2}\right) > 0 \label{smalltaucond}
\end{equation}
for all $0<\tau\leq9/(d+5)$.
\end{lemma}

\begin{proof} The proof will proceed as follows. For $0<\tau\leq9/(d+5)$, we restate the desired inequality \eqref{smalltaucond} as a condition on $\gamma$, \eqref{gammacond}. We then use properties of Bessel functions to establish a lower bound on $\gamma$ in terms of a rational function of $a$; we then show this function satisfies \eqref{gammacond}. We will need to treat the cases of $d\geq 3$ and $d=2$ separately, because the two-dimensional case requires better bounds than we can derive for general $d$.

First note that $b^2=a^2+\tau$, so the inequality \eqref{smalltaucond} is equivalent to
\begin{equation}
\tau > \frac{3a^2(1-\gamma)}{(d+5)\gamma+(d+2)}.
\end{equation}
Using the lower bound on $\tau$ in \eqref{abounds}, we see that the above will hold if
\begin{equation}
\gamma \geq \frac{3(d+2)-a^2(d+5)}{(3+a^2)(d+2)}=:\gamma^*.\label{gammacond}
\end{equation}

We need only show that \eqref{gammacond} holds for all $0<\tau \leq 9/(d+5)$. We will use Taylor polynomial estimates to bound $\gamma$ below by a rational function. From Lemma~\ref{ijbounds}, we have
\begin{align*}
 j_1''(z)&\leq -d_1z+d_2z^3 &\text{on $[0,\sqrt{3(d+2)/(d+5)}]$,}\\
i_1''(z) &\leq d_1z+\frac{6}{5}d_2z^3&\text{on $[0,\sqrt{3}]$,}.
\end{align*}
These bounds apply to $z=ar$ and $z=br$ respectively, when $r\in[0,1]$, as we show below by obtaining bounds on $a^2$ and $b^2$.

To derive our bound on $a^2$, we note that the lower bound of \eqref{abounds} together with our assumption $\tau\leq 9/(d+5)$ implies
\[
 \frac{a^4}{d+2-a^2}<\frac{9}{d+5},
\]
so that 
\[
(d+5)a^4+9a^2-9(d+2)<0.
\]
The lefthand side is increasing with respect to $a^2$ and equals zero when $a^2=3(d+2)/(d+5)$. Hence $a^2<3(d+2)/(d+5)$ and the bound on $j_1''(z)$ holds for $z=ar$ when $\tau\leq 9/(d+5)$. We use these to obtain a further bound:
\begin{align*}
0&\geq\tau-\frac{9}{d+5}\\
&>\frac{a^4}{d+2-a^2}-\frac{9}{d+5}\\
&=\frac{a^2+3}{d+2-a^2}\left(a^2-\frac{3(d+2)}{d+5}\right),
\end{align*}
and so we have $a^2<d$.

To bound $b^2$, we use $b^2=a^2+\tau$ and obtain
\[
b^2=a^2+\tau \leq \frac{3(d+2)}{d+5}+\frac{9}{d+5} = 3,
\]
and so $b^2\leq 3$.

We also have, from \eqref{abounds},
\begin{equation}
\frac{da^2}{d-a^2} > b^2 > \frac{(d+2)a^2}{d+2-a^2},\label{bbounds}
\end{equation}
with the upper bound holding in this regime because $a^2<d$.

We also need the following binomial estimate:
\begin{equation}
 1-\frac{3}{2}x<(1-x)^{3/2} \qquad\text{for $0<x<1$.}\label{32bound}
\end{equation}

Using these bounds, we see
\begin{align*}
\gamma &= \frac{-a^2 j_1^{\prime\prime}(a)}{b^2 i_1^{\prime\prime}(b)} &&\text{by definition \eqref{gammadef}}\\
&\geq \frac{a^2(d_1a-d_2a^3)}{b^2(d_1b+(6/5)d_2b^3)}&&\text{by Lemma~\ref{ijbounds}}\\
&\geq \frac{a^3(d_1-d_2a^2)}{\left(\frac{da^2}{d-a^2}\right)^{3/2}(d_1+(6/5) d_2 \frac{da^2}{d-a^2})} &&\text{by \eqref{bbounds}}\\
&= \left(\frac{d-a^2}{d}\right)^{3/2}\frac{(d-a^2)(1-c_1a^2)}{(d-a^2+(6/5) c_1da^2)}&&\text{writing $c_1=d_2/d_1=5/6(d+4)$}\\
&\geq \left(1-\frac{3a^2}{2d}\right)\frac{(d-a^2)(6(d+4)-5a^2)}{(6d(d+4)-24a^2)} &&\text{by \eqref{32bound},}
\end{align*}
noting that $a^2/d < 1$ and $a^2<3(d+2)/(d+5)$.

Thus we have $\gamma-\gamma^*\geq 0$ if
\[
\left(1-\frac{3a^2}{2d}\right)\frac{(d-a^2)(6(d+4)-5a^2)}{6d(d+4)-24a^2}-\frac{3(d+2)-a^2(d+5)}{(3+a^2)(d+2)}\geq 0,
\]
or, clearing the denominators and writing $x=a^2$, if
\[
(2d-3x)(d-x)\Big(6(d+4)-5x\Big)(3+x)(d+2)-2d\Big(6d(d+4)-24x\Big)\Big(3(d+2)-x(d+5)\Big)\geq 0.
\]
The above polynomial is fourth degree in each of $d$ and $x$ and has the root $x=0$; because we are only interested in its behavior for $x\in(0,3(d+2)/(d+5))$, we may divide by $x$ and work to show the resulting polynomial 
\begin{align*}
P(x,d) &= 24d^4 +60d^3 -120d^2-432d -40 d^3 x -119d^2x -6dx +432x \\
&\qquad+43d^2x^2+113dx^2+54x^2-15dx^3-30x^3
\end{align*}
is nonnegative for $x\in(0,3(d+2)/(d+5)$. This claim is addressed in Lemma~\ref{poly1} for $d\geq3$.

For $d=2$, the function $P(x,2)$ is negative on most of our interval of interest $[0,12/7]$, and so we must improve our lower bound on $\gamma$. The derivation follows that of inequality \eqref{abounds} in the proof of Lemma~\ref{largetau}, as follows. 

By Lemma~\ref{wbounds}, $\ostar>\ainfsq\tau$, where $\ainf\approx1.84118$ is the first zero of $J_1(z)$.  By Chapter~\ref{unitball} we have $\ostar=a^4+a^2\tau$, giving us
\[
\tau \leq \frac{a^4}{\ainfsq- a^2}
\]
Using $b^2=a^2+\tau$, we obtain also a bound on $b^2$:
\[
\quad b^2\leq \frac{\ainfsq a^2}{\ainfsq-a^2}.
\]
Proceeding as before, we deduce
\[
\gamma \geq \left(1-\frac{3a^2}{2\ainfsq}\right)\frac{(\ainfsq-a^2)(36-5a^2)}{36\ainfsq+(6\ainfsq-36)a^2}
\]
with the last again from \eqref{32bound}. So $\gamma-\gamma^*\geq 0$ if
\[
\left(1-\frac{3a^2}{2\ainfsq}\right)\frac{(\ainfsq-a^2)(36-5a^2)}{36\ainfsq+(6\ainfsq-36)a^2}-\frac{12-7a^2}{12+4a^2}\geq 0
\]
or, setting $x=a^2$, if the fourth degree polynomial
\[
Q(x)=\left(1-\frac{3x}{2\ainfsq}\right)(\ainfsq-x)(36-5x)(12+4x)-\Big(36\ainfsq+(6\ainfsq-36)x\Big)(12-7x)
\]
is positive on $[0,12/7]$. This positivity follows from Lemma~\ref{poly2}, completing our proof.
\end{proof}

\begin{lemma}\label{poly1} The polynomial
\begin{align*}
P(x,d) &= 24d^4 +60d^3 -120d^2-432d -40 d^3 x -119d^2x -6dx +432x \\
&\qquad+43d^2x^2+113dx^2+54x^2-15dx^3-30x^3
\end{align*}
is nonnegative for all $x\in(0,3(d+2)/(d+5))$ and integers $d\geq 3$.
\end{lemma}
\begin{proof}
First note that $3(d+2)/(d+5)<3$. We bound $P$ below on the interval $x\in[0,3]$ by taking $x=3$ in terms with negative coefficients and taking $x=0$ in terms with positive coefficients, obtaining
\[
P(x,d)\geq  24d^4 -60d^3-477d^2 -855 d - 810 =: g(d).
\]
The highest order term is $24d^4$, and so $g$ is ultimately positive and increasing in $d$. Note also that
\begin{align*}
g'(d)&= 96 d^3- 180 d^2- 954 d-855\\
g''(d) &= 288d^2-360d-954.
\end{align*}
The function $g''(d)$ is a quadratic polynomial with positive leading coefficient and roots at $d \approx -1.30$ and $2.55$; thus $g'(d)$ is increasing for all $d\geq 3$.  We see that $g'(5)=1875$, so $g$ is increasing for all $d\geq 5$. Finally, $g(7)=6876$, so for all $d\geq7$ we have $g(d)> 0$ and hence $P(x,d)>0$ for all $d\geq 7$ and $x\in[0,3]$.

For $d=3,4,5,6$, we look at the polynomials $P_d(x)=P(x,d)$ directly to show that $P_d(x)>0$ on $[0,3(d+2)/(d+5)]$.  Each $P_d$ is a cubic polynomial in $x$; its first derivative $P_d^\prime(x)$ is quadratic and so the critical points of $P_d(x)$ can all be found exactly.

For $d=4$, $5$, and $6$, direct calculations show $P'_d<0$ on $[0,3(d+2)/(d+5)]$ and $P_d(3(d+2)/(d+5))>0$, so $P_d(x)>0$ on $[0,3(d+2)/(d+5)]$.

For $d=3$, our interval of interest is $[0,15/8]$. We have a critical point \\ $c\approx1.4\in[0,15/8]$, with $P_3'<0$ on $[0,c]$ and $P_3'>0$ on $[c,15/8]$. The critical value $P_3(c)\approx 79$ is positive, so $P_3(x)>0$ on the desired interval $[0,15/8]$.
\end{proof}

\begin{lemma}\label{poly2} The polynomial
\[
Q(x)=\left(1-\frac{3x}{2\ainfsq}\right)(\ainfsq-x)(36-5x)(12+4x)-\Big(36\ainfsq+(6\ainfsq-36)x\Big)(12-7x)
\]
is positive on $[0,12/7]$.
\end{lemma}
\begin{proof}
As in previous cases, $x=0$ is a root of this polynomial, so we examine $g(x):=Q(x)/x$. The derivative $g'(x)$ is a quadratic polynomial, so its roots can be found exactly. We see that $g$ has a critical point $c\approx 1.4$ in $[0,12/7]$, with $g'<0$ on $[0,c]$ and $g'>0$ on $[c,12/7]$. The critical value $g(c)\approx 177.8$ is positive, so $g>0$ on $[0,12/7]$.
\end{proof}

Our final lemma is a simple observation about integrals of monotone and partially monotone functions, which is a special case of more general rearrangement inequalities (see \cite[Chapter 3]{liebloss}).
\begin{lemma}\label{monint}
For any radial function function $F(r)$ that satisfies the partial monotonicity condition \eqref{moncond} for $\Ostar$,
\[
\int_\Omega F\,dx \leq \int_{\Ostar}F\,dx
\]
with equality if and only if $\Omega=\Ostar$.
For any strictly increasing radial function $F(r)$,
\[
\int_\Omega F\,dx \geq \int_{\Ostar}F\,dx
\]
with equality if and only $\Omega=\Ostar$.
\end{lemma}
\begin{proof} Note that $|\Omega|=|\Ostar|$ with $|\Omega\backslash\Ostar|=|\Ostar\backslash\Omega|$. Suppose $F$ satisfies \eqref{moncond} for $\Ostar$. The result follows from decomposing the domain:
\begin{align*}
\int_\Omega F\,dx &=\int_{\Omega\cap\Ostar}F\,dx+\int_{\Omega\setminus\Ostar}F\,dx\\
&\leq \int_{\Omega\cap\Ostar}F\,dx+\sup_{x\in\Omega\setminus\Ostar}|F(x)||\Omega\setminus\Ostar|  \\
&\leq \int_{\Omega\cap\Ostar}F\,dx+\inf_{x\in\Ostar\setminus\Omega}|F(x)||\Ostar\setminus\Omega|&&\text{since $F$ satisfies \eqref{moncond}.}\\
&\leq \int_{\Omega\cap\Ostar}F\,dx+\int_{\Ostar\setminus\Omega}F\,dx\\
&=\int_\Ostar F\,dx.
\end{align*}
Note that if $|\Omega\backslash\Ostar|>0$, either the second inequality or the third is strict by the strict inequality in \eqref{moncond}. If $F$ is strictly increasing, then apply the first part of the Lemma to the function $-F$.
\end{proof}

We can now prove our main result.
\begin{proof}[Proof of Theorem~\ref{thm1}]
By rescaling as described after Lemma~\ref{scaling}, it suffices to prove the theorem for $\Omega$ with volume equal to that of the unit ball, so that $\Ostar$ is the unit ball. We may also translate $\Omega$ as in Lemma~\ref{trialfcn}, which of course leaves the fundamental tone unchanged. Then,
\begin{align*}
\omega &\leq \frac{\int_\Omega N[\rho]\,dx}{\int_\Omega \rho^2\,dx} &&\text{by Lemma~\ref{lemmaboundRC}}\\
&\leq \frac{\int_\Ostar N[\rho]\,dx}{\int_\Ostar \rho^2\,dx}&&\text{by Lemmas~\ref{mondenom}, \ref{monnum}, and~\ref{monint}}\\
&=\ostar,
\end{align*}
by applying the equality condition in Lemma~\ref{lemmaboundRC}. Finally, if equality holds, then $\Omega$ must be a ball, by the equality statement in Lemma~\ref{monint}.
\end{proof}

%% file: 1dim.tex
\chapter{One dimension: the free rod}\label{1dim}

The free rod is the one-dimensional case of the free plate. We do not have an isoperimetric inequality for the free rod, because all connected domains of the same area are now intervals of the same length, identical up to translation. However, we do have a one-dimensional analog of Theorem~\ref{thm2}, which identifies the fundamental tone of the circular free plate. In higher dimensions, we showed the fundamental mode of the free plate under tension had angular dependence; in one dimension, we will show the fundamental mode of the free rod under tension is an odd function about the center point. In this chapter we prove the one-dimensional analog of Theorem~\ref{thm2} and classify most eigenfunctions of the free rod under tension and compression.

We take $\Omega = [-1,1]$; the general case follows from rescaling and translation. The Rayleigh quotient in one dimension is
\[
  Q[u] = \frac{\int_{-1}^1|u^{\prime\prime}|^2 + \tau |u^\prime|^2\,dx}{\int_{-1}^1|u|^2\,dx}.
\]
From this quotient, we obtain the eigenvalue equation
\begin{equation}
        u^{\prime\prime\prime\prime}-\tau u^{\prime\prime} = \omega u \label{1dimDE}
\end{equation}
and the natural boundary conditions
\begin{align}
 u^{\prime\prime}&=0 &\text{at $x=\pm 1$,} \label{1dbcm}\\
 \tau u^\prime - u^{\prime\prime\prime}&=0 &\text{at $x=\pm 1$.}\label{1dbcv}
\end{align}
As in the case of the ball in $\RR^d$, we find the general form of the eigenfunctions by factoring the eigenfunction equation \eqref{1dimDE}. The factorization depends on the sign of the eigenvalue $\omega$ and in some cases on the value of $\tau$ relative to $\omega$. We will therefore treat positive, zero, and negative eigenvalues as separate cases. In each case, we will use the boundary conditions to help further identify form of the eigenfunctions.

First we show we need only consider even and odd solutions, before embarking on the classification. In higher dimensions, we were concerned with angular dependence in our solutions, so it is natural in one dimension to consider the symmetries of the solutions about the center point. Note that if $u(x)$ is an eigenfunction, then by symmetry so is $u(-x)$, with the same eigenvalue $\omega$. The even and odd parts of $u$ can be expressed as
\[
u_o(x)=\frac{u(x)-u(-x)}{2} \quad \text{and} \quad u_e(x)=\frac{u(x)+u(-x)}{2},
\]
which are either solutions of \eqref{1dimDE} with the same eigenvalue, or (in the case that $u$ is purely odd or purely even), one of them is zero everywhere. Because $u(x)$ and $u(-x)$ both satisfy the boundary conditions, $u_e$ and $u_o$ will also satisfy them. Since every eigenfunction is a linear combination of its odd and even parts, it suffices to look only for even and odd eigenfunctions. We will refer to eigenvalues associated with even and odd eigenfunctions as even and odd eigenvalues, respectively.

\section*{Positive eigenvalues}
In the case of positive eigenvalues, the general forms of the eigenfunctions do not depend on the sign of $\tau$. As in the case of the plate, positive eigenvalues correspond to vibrational frequencies of the rod; the corresponding eigenfunctions are then vibrational modes. In higher dimensions and taking $\Omega$ to be the ball, we saw solutions involving Bessel $J_l$ and $I_l$ functions; here the corresponding solutions are trigonometric and hyperbolic trigonometric functions.

\begin{fact} \label{posclass} Given $\tau\in\RR$, a positive number $\omega$ is an eigenvalue if and only if it has an associated eigenfunction $u$ with $\omega$ and $u$ satisfying exactly one of the following:
\begin{enumerate}
 \item [(i)] We have $u= u_o(x)=A\sin{ax}+B\sinh{bx}$ with $A$, $B$ real nonzero constants satisfying
\[
  \frac{B}{A}=\frac{a^2\sin a}{b^2\sinh b}
\]
and $a$, $b$ positive numbers satisfying $a^2b^2=\omega$, $b^2-a^2=\tau$, and $W_o(a) = a^3\sin a\cosh b - b^3\cos a\sinh b=0$.
\item[(ii)] We have $u=u_e(x)=C\cos{ax}\cosh{bx}$ with $C$, $D$ real nonzero constants satisfying
\[
\frac{D}{C}=\frac{a^2\cos a}{b^2\cosh b}.
\]
and $a$, $b$ positive numbers satisfying $a^2b^2=\omega$, $b^2-a^2=\tau$, and $W_e(a) = a^3\cos a\sinh b + b^3\sin a\cosh b = 0$.
\end{enumerate}
\end{fact}
\noindent We will also prove a theorem identifying the fundamental tone of the free rod under tension::
\begin{thm2pp}
 For all $\tau>0$, we have that the even and odd eigenvalues are interlaced. In particular, the fundamental mode is an odd function.
\end{thm2pp}

We prove Fact~\ref{posclass} first.
\begin{proof}[Fact~\ref{posclass}]
 We begin by factoring the eigenvalue equation to establish the general form of the even and odd solutions. We will then use the boundary conditions to obtain the desired values for $B/A$ and $D/C$ and show that they are nonzero and well-defined.

When $\omega$ is positive, regardless of the value of $\tau$, we can factor the eigenvalue equation \eqref{1dimDE} as
\[
  (d_x^2 + a^2)(d_x^2 - b^2)u = 0,
\]
where $\omega = a^2b^2$ and $\tau = b^2-a^2$. Since $\omega >0$, $a$ and $b$ must be nonzero as well; we take them to be positive. Recall that we need only look for even and odd solutions; all others will be linear combinations of these.

Solutions to $(d_x^2 + a^2)u=0$ are of the form $e^{\pm a i x}$; we express them as the odd and even trigonometric functions $\sin xa$ and $\cos ax$. Solutions $(d_x^2 - b^2)u=0$ are of the form $e^{\pm b x}$, or $\sinh bx$ and $\cos bx$. Any solution $u$ to \eqref{1dimDE} is then a linear combination of these terms; since we need only consider odd and even functions, we obtain
\begin{equation}
  u_o(x) = A\sin ax +B\sinh bx  \quad \text{and} \quad  u_e(x) = C\cos ax +D\cosh bx.
\end{equation}
Applying the boundary conditions to $u_o$ and $u-e$, we obtain:
\begin{align}
a^2A\sin a = b^2B\sinh b &\quad\text{and}\quad ab^2A\cos a=a^2bB\cosh a\label{AB1}\\
a^2C\cos a = b^2D\cosh b&\quad\text{and}\quad-ab^2C\sin a=a^2bD\sinh a \label{CD1}
\end{align}

We will now show that for odd solutions $u_o$, both $A$ and $B$ are nonzero, and that for nonconstant even solutions $u_e$, both $C$ and $D$ are nonzero.

If $A=0$ but $B$ nonzero, then the odd eigenstate is $u_o(x)=B\sinh(bx)$ and the boundary conditions \eqref{AB1} become
\[
b^2B\sinh(b)=0 \quad\text{and}\quad a^2bB\cosh(a)=0.
\]
Since $b$ is positive and $B$ is nonzero, the first condition cannot be satisfied.

If $B=0$ but $A$ nonzero, then $u_o(x)=A\sin(a)$ and the boundary conditions \eqref{AB1} become
\[
-Aa^2\sin(a) = 0 \quad\text{and}\quad ab^2A\cos(a) = 0.
\]
Since $a$, $b$, and $A$ are nonzero and $\sin(a)$ and $\cos(a)$ cannot be simultaneously zero, we cannot satisfy both conditions.

If $C=0$, our even eigenstate is $u_e(x)= D\cosh(bx)$, and the boundary conditions \eqref{CD1} become
\[
b^2D\cosh(b)=0 \quad\text{and}\quad a^2bD\sinh(b)=0.
\]
Again, $b>0$, so the only way both equations can be satisfied is if $D=0$, giving us $u_e(x)=0$ a constant eigenfunction.

If $D=0$, we have $u_e(x)= C\cos(ax)$ and the boundary conditions \eqref{CD1} become
\[
a^2C\cos(a)=0 \quad\text{and}\quad -ab^2C\sin(a)=0.
\]
Since $a$, $b$ are nonzero and $\sin(a)$ and $\cos(a)$ cannot be simultaneously zero, we can only satisfy both conditions if $C=0$, giving us a constant eigenfunction $u_e(x)=0$.

Note that since $b>0$, we have $\sinh b$ positive. To obtain the ratios in Fact~\ref{posclass}, we solve the first boundary condition in each of \eqref{AB1} and \eqref{CD1}.

It remains to show that $a$ and $b$ satisfy the conditions $W_o(a)=0$ for the odd eigenfunction $u_o$ and $W_e(a)=0$ for the even eigenfunction $u_e$.

We first  consider the odd solution $u_o(x) = A\sin ax  + B\sinh bx$. As in the higher-dimensional case, the two boundary conditions $Mu=u_o^{\prime\prime}=0$ and $Vu=\tau u_o^\prime-u_o^{\prime\prime\prime}=0$ give us a pair of homogeneous equations which are linear in $A$ and $B$. Thus we must have the determinant vanish:
\[
0 = (-a^2\sin a)(-a^2b \cosh b)-(b^2\sinh b)(ab^2\cos a)
\]
As $a$ and $b$ are nonzero, this simplifies to
\[
0 = a^3\sin a\cosh b-b^3\cos a\sinh b =:W_o(a).
\]
Recall $b^2=\tau+a^2$. Thus, for a given $\tau$, our candidate $u_o(x)$ is a solution to the boundary value problem if $a$ satisfies $W_o(a)=0$ and $A/B$ is taken to satisfy either (and hence both) of the boundary conditions.

Now we consider the even solution $u_e(x) = C\cos ax  + D\cosh bx$. The constants $C$ and $D$ are assumed to both be nonzero by Fact~\ref{posclass}and are subject to the constraints \eqref{CD1}. Again we must have the determinant vanish:
\[
0 =-a^2\cos a(\mp a^2b\sinh b)-b^2\cosh bab^2(\mp \sin a).
\]
That is,
\[
0= a^3\cos a\sinh b + b^3\sin a\cosh b=:W_e(a).
\]
Then our candidate $u_e(x)$ is a solution to the boundary value problem if $a$ is a root of $W_e$ and $C/D$ is taken to satisfy either (and hence both) of the boundary conditions.
\end{proof}

The proof of Theorem~2$^{\prime\prime}$ requires the following lemma:

\begin{lemma}\label{fefo}Let $\tau>0$ and write $b = \sqrt{a^2+\tau}$. The nontrivial zeros of the functions
\begin{align*}
W_e(a) &= a^3\cos a\sinh b + b^3\sin a\cosh b\quad\text{and}\\
W_o(a) &= a^3\sin a\cosh b - b^3\cos a\sinh b.
\end{align*}
on $[0,\infty)$ are distinct and are interlaced. In particular, $W_o$ has a zero in $(0,\pi /2)$, and the first nontrivial zero of $W_e \in (\pi/2,\infty)$.
\end{lemma}
\begin{proof} We consider $W_e$ first. Recall $\tau=a^2+b^2$; viewing $a$ as a variable independent of $\tau$, we differentiate and obtain
\[
 \frac{d b}{d a}=\frac{-a}{b}.
\]
Then
\begin{align*}
 W_e^\prime(a)&=3a^2\cos a\sinh b-a^3\sin a\sinh b-a^4b^{-1}\cos a \cosh b \\
&\qquad -3ab\sin a \cosh b + b^3\cos a \cosh b - ab^2\sin a \sinh b\\
&=-a^3\sin a\sinh b-3ab\sin a \cosh b- ab^2\sin a \sinh b\\
&\qquad +(\tau^2+2a^2\tau)b^{-1}\cos a \cosh b +3a^2\cos a\sinh b.
\end{align*}
Recall $b$ and hence $\sinh b$ are positive. For $a$ in the intervals $(\pi/2+2k\pi, \pi+2k\pi)$ with $k$ any nonnegative integer, we have $\sin a > 0$ and $\cos a <0$, giving us $W_o^\prime(a)<0$. When $a\in(3\pi/2+2k\pi, 2\pi+2k\pi)$, the signs are reversed, and $W_o^\prime(a)>0$.

Similarly, we find that $W_e(a)>0$ on $[2k\pi,\pi/2+2k\pi]$ and $W_e(a)<0$ on $[\pi+2k\pi,3\pi/2+2k\pi]$. Therefore, $W_e$ has exactly one zero in each interval $(\pi/2+k\pi, (k+1)\pi)$ for $k$ a nonnegative integer, and is nonzero everywhere else

Now consider $W_o$. Differentiating and simplifying, we obtain
\begin{align*}
 W_o^\prime(a)= a^3\cos a \cosh b +3ab\cos a \sinh +ab^2\cos a \cosh b\\
3a^2\sin a \cosh b+(\tau^2+2a^2\tau)\sin a\sinh b.
\end{align*}
Examining the sign of the trigonometric terms as before, we see $W_o^\prime(a)>0$ on $(2k\pi,\pi/2+2k\pi)$ and $W_o^\prime(a)<0$ on $(\pi + 2k\pi,3\pi/2+2k\pi)$ for $k$ a nonnegative integer. We also have $W_o(a)>0$ on $[\pi/2+2k\pi, \pi+2k\pi]$ and $W_o(a)<0$ on $[3\pi/2+2k\pi, 2\pi+2k\pi]$.  Thus, $W_o$ has exactly one zero on $(k\pi,\pi/2+k\pi)$ for each nonnegative integer $k$, and is nonzero everywhere else. \end{proof}

We can now prove Theorem~2$^{\prime\prime}$.
\begin{proof}[Proof of Theorem~2$^{\prime\prime}$]
We see by examining the Rayleigh quotient that for $\tau\geq 0$, all eigenvalues are nonnegative. Furthermore, the eigenvalue $\omega=0$ corresponds to the constant eigenfunction and has multiplicity one; all higher eigenvalues are positive.

By Fact~\ref{posclass}, positive eigenvalues are given by $\omega=a^2(a^2+\tau)$, with $a$ a root of $W_e$ or $W_0$. It is clear $\omega$ is increasing in $a$ for $\tau>0$, so the eigenvalues for odd and even modes are interlaced if the zeroes of $W_e$ and $W_o$ are. We have this interlacing by Lemma~\ref{fefo}. Also by this lemma, $W_o$ has the first nontrivial root in $(0,\pi/2)$, and so the fundamental mode is odd.
\end{proof}

As in the higher-dimensional case, Theorem~2$^{\prime\prime}$ only considers positive tension. If $\tau>0$, the Rayleigh quotient is nonnegative, and so we have only nonnegative eigenvalues. Furthermore, the eigenvalue $\omega=0$ is nondegenerate for positive $\tau$.

Our work for negative tension will show that in this regime the eigenvalues often cross, and so the theorem does not hold here.

\section*{Zero eigenvalues}
An eigenvalue of zero corresponds to non-vibrational translation of the rod. Other than the constant eigenstate $u=c$, these eigenstates are only possible when $\tau$ is nonpositive. When tension vanishes, we have an odd eigenstate and $\omega=0$ is of multiplicity 2. For negative tension, the rod is under compression, and the eigenvalue $\omega=0$ is only degenerate when $\sqrt{|\tau|}$ is a root of $\sin x$ or $\cos x$, and in these cases has multiplicity 2.

\begin{fact}\label{zeroclass} For all real values of $\tau$, $\omega=0$ is an eigenvalue and the constant function is an associated eigenfunction. The eigenvalue $\omega=0$ is degenerate if and only if it has an associated nonconstant eigenfunction $u$ satisfying one of the following:
\begin{enumerate}
 \item [(i)] (Zero tension) The parameter $\tau = 0$ and we have $u=u_o(x)=Ax$ where $A$ is a nonzero real constant.
 \item [(ii)] (Negative tension) The parameter $\tau<0$ and $u$ satisfies one of the following:
 \begin{enumerate}
 \item We have $u=u_o(x)=B\sin{ax}$ where $a=\sqrt{|\tau|}$ and $B$ is a nonzero real constant, and $\sin\sqrt{|\tau|}=0$.
 \item We have $u=u_e(x)=D\cos{ax}$ where $a=\sqrt{|\tau|}$ and $D$ is a nonzero real constant, and $\cos\sqrt{|\tau|}=0$.
 \end{enumerate}
 \end{enumerate}
For all other values of $\tau$, the eigenvalue $\omega=0$ is nondegenerate.
\end{fact}

\begin{proof}
Note that the form of the solutions depends on the sign of $\tau$; we will treat each as a separate case. The positive and zero tension cases are fairly easy to solve; we only need to factor \eqref{1dimDE} when $\tau<0$.

Positive tension.  When $\tau>0$,  the numerator of the Rayleigh Quotient $Q$ can be zero only when $u_x=0$ almost everywhere. Functions with derivative of zero are themselves constant, and so the only solution to \eqref{1dimDE} with $\omega=0$ and $\tau=0$ is the constant solution $u=C$. Note that in this case we have no odd solutions.

Zero tension. Here we have $\omega$ and $\tau$ both zero; the eigenvalue equation becomes
\[
 u^{\prime\prime\prime\prime}=0,
\]
which has general solution
\[
u(x) = Ax^3+Bx^2+Cx+D.
\]
Our first boundary condition, $u''(\pm 1)=0$, gives us
\[
0=u''(\pm 1)=\pm 6A+2B,
\]
so we must have $A$ and $B$ both zero. The second boundary condition then yields $u'''(\pm 1)=0$, and gives us no additional information. Thus, $u(x)=Cx+D$ is a solution to the boundary problem for any $C$, $D\in \RR$. The even solutions are then the constant functions $u=D$; odd solutions are given by $u=Cx$.

Negative tension. When $\tau<0$ we have a free rod under compression; the proof proceeds similarly to the positive eigenvalue case, Fact~\ref{posclass}.

As in the classification of positive eigenvalues, we factor the eigenfunction equation to get the general form of the solution and then apply the boundary conditions to get constraints on the coefficients $A$, $B$, $C$, and $D$, and on possible values for $\tau$.

Since $\omega$ is zero, the eigenvalue equation \eqref{1dimDE} factors as
\[
d_x^2(d_x^2 - \tau)u = 0,
\]
so we must consider solutions to $u''=0$ and $u''-\tau u=0$. Solutions to the first factor are of the form $u=Ax + B$. The solutions of $u''-\tau u=0$ are all of the form $A\cos(ax)+B\sin(ax)$, with $a=\sqrt{|\tau|}$. The even and odd solutions to \eqref{1dimDE} are then of the form
\[
u_o(x)=Ax+B\sin(ax) \quad \text{and}\quad u_e(x) = C+D\cos(ax)
\]
and must satisfy the boundary conditions
\begin{align*}
-Ba^2\sin(a)&=0 &\text{and}\quad &A\tau +B(a\tau+a^3)\cos(a)=0 &\text{for $u_o$,}\\
-Da^2\cos(a)&=0 &\text{and}\quad &-D(a\tau+a^3)\sin(a)=0&\text{for $u_e$.}
\end{align*}
Consider the odd solutions first. Because $a=\sqrt{|\tau|}$, we have $a\tau+a^3=0$, and so the second condition for $u_o$ reduces to $A=0$. We then have $u_o(x)=B\sin(ax)$; we assume $B$ is nonzero so that $u_o$ is not even. The first boundary condition then can only be satisfied if $a$ or $\sin a$ vanish. The parameter $\tau$ is nonzero, so we must have $\sin(a)=0$.

Therefore odd solutions are all of the form $B\sin ax$ and exist only when $|\tau|$ is a root of $\sin\sqrt{z}$.

Now consider the even solutions. The second condition for $u_e$ simplifies by $a\tau+a^3=0$ to $0=0$, giving us no information. The first condition requires $D=0$ or $\cos(a)=0$. Hence, all even solutions will be of the form $u_e(x)=C+D\cos(ax)$ with $\cos(a)=0$. The boundary conditions give us no relationship between the coefficients, so $C$ and $D$ may be any real numbers. These solutions will only occur only when $|\tau|$ is a root of $\sin\sqrt{z}$.
\end{proof}

\section*{Negative eigenvalues}
All negative eigenvalues and their associated eigenfunctions correspond to the buckling of the free rod under compression. The Rayleigh quotient can only be negative when $\tau$ is negative, which corresponds to a rod under compression rather than tension.

Finding the solutions for negative eigenvalues is more complicated than either the positive or zero eigenvalue cases. The factorization of \eqref{1dimDE}, and hence the general form of the solutions, depends on a relationship between $\tau$ and $\omega$. There are three regimes, named for the behavior of the solutions there, and we will prove our results in each regime separately. Unfortunately, we cannot complete our classification in the case of the hyperbolic regime; we provide only a necessary condition for a number $\omega$ to be an eigenvalue in this regime.

\begin{fact} \label{negclass} Given $\tau$, a number $\omega\geq -\tau^2/4$ is a negative eigenvalue if and only if it has an associated eigenfunction $u$ with $u$, $\omega$ satisfy exactly one of the following conditions:
\begin{enumerate}
 \item [(i)] (trigonometric regime) $\omega>-\tau^2/4$ and $u$ satisfies one of the following:
\begin{enumerate}
\item We have $u= u_o(x)=A\sin{ax}+B\sin{bx}$ with $A$, $B$ real nonzero constants satisfying
\[
\frac{A}{B}=-\frac{b^2\sin(b)}{a^2\sin(a)} \quad\left(\text{or} \quad \frac{A}{B}=-\frac{a\cos(b)}{b\cos(a)}\quad\text{when $\sin(a)=0$}\right).
\]
and $a$, $b$ positive numbers satisfying $-a^2b^2=\omega$, $-a^2-b^2=\tau$, and $W_o(a)=a^3\sin(a)\cos(b)-b^3\sin(b)\cos(a)$.
\item We have $u= u_e(x)=D\cos{ax}+D\cos{bx}$ with $C$, $D$ real nonzero constants satisfying
\[
\frac{C}{D}=-\frac{b^2\cos(b)}{a^2\cos(a)}\quad\left(\text{or} \quad \frac{C}{D} =-\frac{a\sin(b)}{b\sin(a)}\quad\text{when $\cos(a)=0$}\right).
\]
and $a$, $b$ positive numbers satisfying $a^2b^2=\omega$, $b^2-a^2=\tau$, and $W_e(a)=b^3\cos(b)\sin(a)-a^3\cos(a)\sin(b)$.
\end{enumerate}
\item[(ii)] (degenerate regime) $\omega = -\tau^2/4$ and
\[
 u=u_e(x)=C\cos{ax}+Dx\sin{ax},
\]
where $a=\sqrt{|\tau|/2}\approx 1.1394$ and is the positive root of $\sin{2a}-2a/3=0$, with $C$, $D$ real nonzero constants satisfying
\[
  \frac{C}{D}=\frac{2\cos(a)-a\sin(a)}{a\cos(a)}\approx -0.4174.
\]
Here $\tau\approx -2.5967$ and $\omega\approx-1.6856$.
\end{enumerate}
\emph{(hyperbolic regime)} A number $\omega < -\tau^2/4$ is an eigenvalue only if it has an associated eigenfunction $u$ satisfying one of the following:
\begin{enumerate}
\item[(a)] We have $u= u_o(x)=A\cos{ax}\sinh{bx}+B\sin{ax}\cosh{bx}$ with $A$, $B$ real constants with $B$ nonzero and $a$, $b$ positive numbers satisfying\\ $-(a^2+b^2)^2=\omega$ and $2(b^2-a^2)=\tau$.
\item[(b)] We have $u= u_e(x)=C\sin{ax}\sinh{bx}+D\cos{ax}\cosh{bx}$ with $C$, $D$ real constants with $D$ nonzero and $a$, $b$ positive numbers satisfying\\ $-(a^2+b^2)^2=\omega$ and $2(b^2-a^2)=\tau$.
\end{enumerate}
\end{fact}

In the proof of each case, we begin by factoring the eigenvalue equation \eqref{1dimDE}, and then apply the boundary conditions to obtain conditions on the coefficients $A$ through $D$, and in the degenerate case, on the values of $\tau$ for which there is a solution.

\subsection*{The trigonometric regime: $-\tau^2/4<\omega<0$}
In the trigonometric regime, the eigenfunctions are linear combinations of trigonometric functions. Here the characteristic function for \eqref{1dimDE} is a fourth-degree polynomial with only purely imaginary roots.

\begin{proof}[Proof for the trigonometric regime]
In this case, \eqref{1dimDE} may be factored as
\[
(d_x^2+a^2)(d_x^2+b^2)u=0.
\]
with $a^2b^2=-\omega$, $-a^2-b^2=\tau$, and $a$ and $b$ taken to be nonnegative. Because $\omega<0$, neither $a$ nor $b$ can be zero. The even and odd solutions will then have the forms
\begin{equation}
 u_o(x)=A\sin(ax)+B\sin(bx) \quad\text{and}\quad u_e(x)=C\cos(ax)+D\cos(bx) \label{compression1}
\end{equation}
with boundary conditions:
\begin{align}
&-a^2A\sin(a)-b^2B\sin(b)=0 &\text{and~} &-ab^2A\cos(a)-a^2bB\cos(b)=0,\label{Codd}\\
&-a^2C\cos(a)-b^2D\cos(b)=0 &\text{and~} &ab^2C\sin(a)+a^2bD\sin(b)=0. \label{Ceven}
\end{align}
If $A=0$, then $-b^2B\sin(b)=0$ and $-a^2bB\cos(b)=0$; the only way to satisfy both conditions is if $B=0$.

If $B=0$, then $-a^2A\sin(a)=0$ and $-ab^2A\cos(a)=0$, and so $A=0$.

If $C=0$, then $-b^2D\cos(b)=0$ and $a^2bD\sin(b)=0$, thus $D=0$.

If $D=0$, then $-a^2C\cos(a)=0$ and $ab^2C\sin(a)=0$, requiring $C=0$.

For a given even or odd solution, given $a$ and $\tau$ (and hence $b$), we must turn to the boundary conditions to determine $A/B$ and $C/D$. It is these ratios that determine the form of the solutions, since multiplication by a nonzero constant does not change the eigenvalue or affect whether or not the boundary conditions are satisfied.

If $A$ and $B$ are nonzero, then
\[
\frac{A}{B}=-\frac{b^2\sin(b)}{a^2\sin(a)} \quad\left(\text{or} \quad =-\frac{a\cos(b)}{b\cos(a)}\quad\text{when $\sin(a)=0$}\right).
\]
If $C$ and $D$ are nonzero, then
\[
\frac{C}{D}=-\frac{b^2\cos(b)}{a^2\cos(a)}\quad\left(\text{or} \quad =-\frac{a\sin(b)}{b\sin(a)}\quad\text{when $\cos(a)=0$}\right).
\]

Armed with this information, let us consider possible eigenvalue crossings. Given $\tau$, suppose there exist both $u_e$ and $u_o$ nontrivial solutions as given in \eqref{compression1}, both with the same associated  eigenvalue, $\omega=-a^2b^2$. Since $a$ and $b$ are completely determined by $\tau$ and $\omega$, we know $A/B$ and $C/D$ from above, and so the solutions are determined up to multiplication by a constant.

By \eqref{Ceven}, we have $\cos(a)=0$ if and only if $\cos(b)=0$; similarly $\sin(a)=0$ if and only if $\sin(b)=0$ by \eqref{Codd}.
Now, $\cos(a)=\cos(b)=0$ implies $a=(2k+1)\pi/2$ and $b=(2l+1)\pi/2$ for some nonnegative integers $k$, $l$. Since $b^2=-a^2-\tau$, solving for $\tau$ yields
\[
\tau=-\frac{\pi^2}{4}\Big((2l+1)^2+(2k+1)^2\Big).
\]
These values for $a$ and $b$ satisfy the boundary conditions for both the even and odd solutions, so the odd and even eigenvalues will coincide at these values of $\tau$. Therefore, we expect the eigenvalue curves cross at these $\tau$ values.

Similarly, $\sin(a)=\sin(b)=0$ implies $a=k\pi$ and $b=l\pi$ for positive integers $k$, $l$. Again we solve for $\tau$, obtaining $\tau=-\pi^2(l^2+k^2)$. Again, these $a$ and $b$ satisfy both sets of boundary conditions, and so the eigenvalues would again coincide at these $\tau$ values. This is illustrated by Figure~\ref{crosspic}.
\begin{figure}\label{crosspic}
\centering
\includegraphics{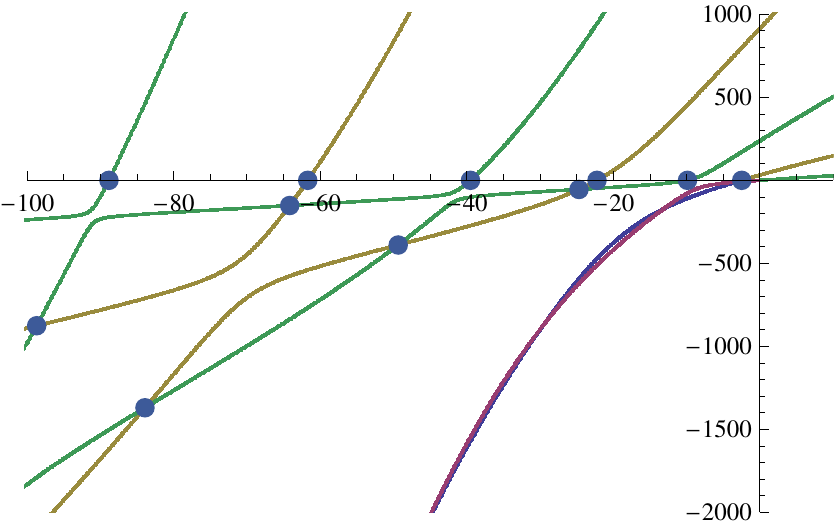}
\caption{Eigenvalue curves with expected crossing points. Green curves are odd eigenvalue branches; gold curves are even eigenvalue branches. The purple and blue curves are possible odd and even eigenvalue branches, respectively, in the hyperbolic regime. These curves were computed under the assumption that all linear combination coefficients are nonzero.} 
\end{figure}

The $\tau$ values for which $\omega=0$ is degenerate are also illustrated in Figure~\ref{crosspic}. These degeneracies occur where even and odd eigenvalue curves cross the line $\omega=0$, which is the eigenvalue curve for the constant eigenfunction.
\end{proof}

\subsection*{The degenerate regime: $\omega=-\tau^2/4$}
In the degenerate regime, the characteristic function has zero as a double root. We will show that there is an eigenvalue with an eigenfunction in the degenerate regime for only one value of $\tau$; that is, the eigenvalue branches only cross from the trigonometric regime to the hyperbolic regime once, when $\omega\approx -1.6856$ and $\tau \approx -2.5966$. The eigenfunction in this case is even.
\begin{proof}[Proof for the degenerate regime]
In this case, \eqref{1dimDE} factors as
\[
(d_x^2+a^2)^2u =0
\]
With $a^2=-\tau/2$. Since $\tau<0$, $a=\sqrt{|\tau|/2}$ is strictly positive. By elementary differential equation theory, the linearly independent solutions of the differential equation \eqref{1dimDE} are $\cos(ax)$, $\sin(ax)$, $x\cos(ax)$, and $x\sin(ax)$. As noted previously, we need only consider even and odd linear combinations of these, namely,
\[
u_o(x)=A\sin(ax)+Bx\cos(ax)  \quad \text{and} \quad  u_e(x)=C\cos(ax)+Dx\sin(ax).
\]
The boundary conditions then yield the relations
\begin{align*}
-a^2A\sin(a)-2aB\sin(a)-a^2B\cos(a) &= 0\\
-a^3A\cos(a)+a^2B\cos(a)+a^3B\sin(a) &=0\\
-a^2C\cos(a)+2aD\cos(a)-a^2D\sin(a) &= 0\\
a^3C\sin(a)+a^2D\sin(a)-a^3D\cos(a) &=0
\end{align*}

We will use these to find the values of $a$ (and thus $\tau$ and $\omega$) and the ratios of the coefficients that give us solutions to the boundary problems.

We begin by proving that if $u_o$ (resp $u_e$) is a nonconstant solution of the boundary problem, both $A$ and $B$ (resp. $C$ and $D$) must be nonzero.

For the odd eigenfunction, suppose $A=0$ but $B\neq0$; we seek a contradiction. The boundary conditions give us $-2a\sin(a)-a^2\cos(a)=0$ and $a^2\cos(a)+a^3\sin(a)=0$. Since $a>0$, we obtain $-2\sin(a)=a\cos(a)$ and $\cos(a)=-a\sin(a)$; thus $2\sin(a)=a^2\sin(a)$. So either $2=a^2$ or $\sin(a)=0$.

If $\sin(a)=0$, then $-a^2D\cos(a)=0$. But each term must be nonzero, so this is impossible, and $a^2=2$. We still must have $\cos(a)=-a\sin(a)$, but we quickly see this is impossible if $a=\sqrt{2}$.

Now suppose $B=0$ but $A\neq0$; we seek a contradiction. Boundary conditions and $a>0$ yield $\sin(a)=0=a\cos(a)$. Since $a\neq0$, both $\sin(a)$ and $\cos(a)$ must be zero, which is impossible.

We consider the even eigenfunction next. Suppose $D\neq0$ but $C=0$; we seek a contradiction.
The boundary conditions give us $2\cos(a)=a\sin(a)$ and $\sin(a)=a\cos(a)$, and combining these we obtain $2\cos(a)=a^2\cos(a)$. So either $2=a^2$ or $\cos(a)=0$.

If $\cos(a)=0$, then the first boundary condition becomes $-a^2B\sin(a)=0$. But $a\neq 0$, $B\neq 0$, and $\sin(a)$ and $\cos(a)$ cannot both be zero, so it is impossible to satisfy the boundary conditions when $\cos(a)=0$.

We must then have $a^2=2$. The boundary conditions still require $\sin(a)=a\cos(a)$, but a simple calculation shows that $a=\sqrt{2}$ does not satisfy this equality.

Now suppose $C\neq0$ but $D=0$; we again seek a contradiction. The boundary conditions then yield $-a^2\cos(a)=0=a^3\sin(a)$, so then $-\cos(a)=0$ and $a\sin(a)=0$. Since $a$ is strictly positive, we must have both $\cos(a)=0$ and $\sin(a)=0$, which is impossible.

Now that we have shown that $A/B$ and $C/D$ must be defined and nonzero if $u_o$ and $u_e$ are to be nonconstant solutions, the boundary conditions give us the following conditions on $a$:
\begin{align*}
\frac{2\sin(a)+a\cos(a)}{-\sin(a)} &=\frac{\cos(a)+a\sin(a)}{\cos(a)}&\text{for $u_o$,}\\
\frac{2\cos(a)-a\sin(a)}{\cos(a)}&=\frac{\sin(a)-a\cos(a)}{-\sin(a)} &\text{for $u_e$.}
\end{align*}
These simplify to
\begin{align*}
3\cos(a)\sin(a)+a&=0 &\text{for $u_o$,}\\
3\cos(a)\sin(a)-a&=0 &\text{for $u_e$.}
\end{align*}
These can be rewritten, using the double-angle formula, as
\begin{align*}
\sin(2a)+2a/3&=0 &\text{for $u_o$,}\\
\sin(2a)-2a/3&=0 &\text{for $u_e$.}
\end{align*}
It is easy to see graphically that the latter has only the solution $a=0$. The former has two nonnegative solutions, $a=0$ and $a\approx 1.13943$. We therefore have no odd solutions in this case, and exactly one even solution with $a\approx 1.13943$. In particular, this means that the curves given by the eigenvalues plotted against $\tau$ may cross the curve $\omega=-\tau^2/4$ for $\tau<0$ only once, at $\tau \approx -2.5966$.
\end{proof}

\subsection*{The hyperbolic regime: $-\tau^2/4>\omega$}
In this case, our characteristic equation $r^4-\tau r-\omega=0$ has non-real, non-purely-imaginary complex roots; the real part of each of these roots then gives us a hyperbolic function in the solution.  This is the most difficult case, and we are only able to obtain a partial result.

\begin{proof}[Proof for the hyperbolic regime]
Elementary differential equation theory gives us that the general solutions will be of the forms $\cos(ax)\sinh(bx)$, $\sin(ax)\sinh(bx)$, $\cos(ax)\cosh(bx)$, and $\sin(ax)\cosh(bx)$, with $\omega=-(a^2+b^2)^2$ and $\tau=2b^2-2a^2$; $a$ and $b$ are taken to be positive as usual. As before we need only consider the even and odd solutions
\begin{align*}
u_o(x)&=A\cos(ax)\sinh(bx)+B\sin(ax)\cosh(bx)\quad\text{and}\\
u_e(x)&=C\sin(ax)\sinh(bx)+D\cos(ax)\cosh(bx).
\end{align*}

Their associated boundary conditions are
\begin{align}
A\left(\tau\cos(a)\sinh(b)-4ab\sin(a)\cosh(b)\right)&=-B\left(\tau\sin(a)\cosh(b)+4ab\cos(a)\sinh(b)\right)\label{hypov}\\
A(a\cos(a)\cosh(b)+b\sin(a)\sinh(b))&=-B(a\sin(a)\sinh(b)-b\cos(a)\cosh(b))\label{hypom}\\
C\Big(\tau\sin(a)\sinh(b)+4ab\cos(a)\cosh(b)\Big)&=-D\left(\tau\cos(a)\cosh(b)-4ab\sin(a)\sinh(b)\right)\label{hypev}\\
C(a\sin(a)\cosh(b)-b\cos(a)\sinh(b))&=-D(a\cos(a)\sinh(b)+b\sin(a)\cosh(b))\label{hypem}.
\end{align}

We are able to show that the constants $B$ and $D$ must be nonzero.

Suppose $B=0$ but $A\neq0$. The boundary conditions \eqref{hypom} and \eqref{hypov} can both be solved for $\tan(a)$; combining the resulting equations we obtain $\tanh^2(b)=-\frac{4a^2}{\tau}$. Recall $\tau=2b^2-2a^2$ is negative; thus we have
$\tanh^2(b)=2+\frac{4b^2}{|\tau|}$.

Set $D=0$ while $C\neq 0$; then by the same process as above we again find $\tanh^2(b)=2+\frac{4b^2}{|\tau|}$.

The function $\tanh^2(x)$ is nonnegative for all $x$ and satisfies $\tanh^2(x)<1$ for all $x$. Thus the equation $\tanh^2(b)=2+\frac{4b^2}{|\tau|}$ has no solution $b>0$ for any $\tau<0$. Therefore is impossible to take $B=0$ or $D=0$ while $A$ or $C$ are nonzero.
\end{proof}

\begin{rmk}
Applying the same process in the case of $A=0$ or $C=0$ yields instead $\coth^2(b)=2+\frac{4b^2}{|\tau|}$, which has a solution $b$ for every $\tau<0$. Furthermore, numerical investigations in Mathematica suggest that in these cases the boundary conditions can be simultaneously satisfied for many values of $\tau$. 
\end{rmk}

\section*{Cascading}
Plotting the eigenvalues of the rod under compression against $\tau$, an interesting feature appears. Each odd eigenvalue branch in the trigonometric regime stays ``close'' to a line $\tau\mu_k$, where $\mu_k$ is a Neumann eigenvalue with an odd eigenfunction, and then sharply drops down and travels along the line $\tau\mu_{k+1}$ corresponding to the next odd Neumann eigenvalue. Similar behavior occurs with the even eigenvalue branches and the corresponding even Neumann eigenvalues. This is very much like the cascading phenomenon exhibited by Schr\"odinger eigenvalues discussed in \cite{simon}. 

 \begin{figure}[th!]
  \includegraphics{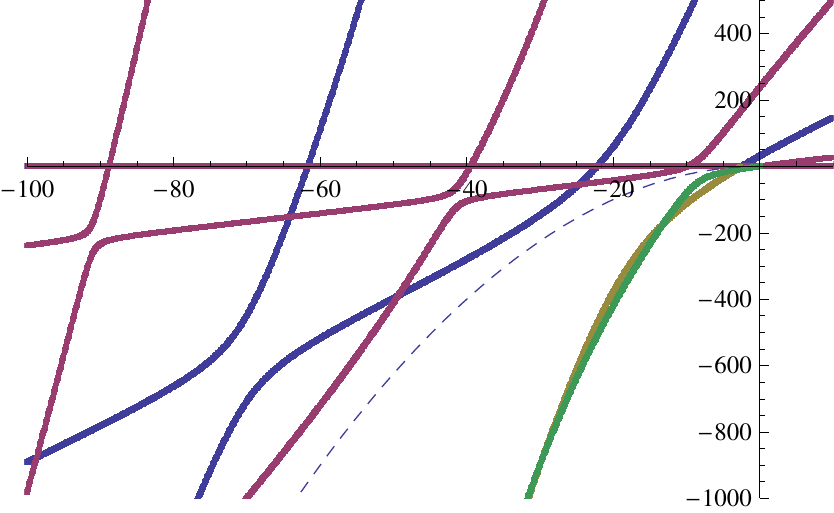}
\caption{Eigenvalues of the free rod under compression. Blue curves correspond to eigenvalues of even functions; purple curves indicate the eigenvalues of odd functions. The dashed curve is the regime boundary $\omega=-\tau^2/4$. The green and gold curves are possible odd and even eigenvalue branches, respectively, in the hyperbolic regime. These curves were computed under the assumption that all linear combination coefficients are nonzero.} 
 \end{figure}

%% file: futuredirec.tex
\chapter{Future directions}\label{furtherdirec}

In this chapter we examine several problems related to the free plate isoperimetric problem. One can generalize the plate Rayleigh quotient to plates made of material having a nonzero Poisson's Ratio; unfortunately, the proofs of the theorems cannot be generalized so far.  Other known isoperimetric inequalities for membranes can be considered for the corresponding plate problems, namely considering a Szeg\H o-type problem for the free plate and a PPW-type problem for the clamped plate. Finally, it may also be worth investigating the plate isoperimetric problems in non-Euclidean spaces.

\section*{Poisson's Ratio}
One generalization of the free plate problem is to account for Poisson's Ratio, a property of the material of the plate that describes how a rectangle of the material stretches or shrinks in one direction when stretched along the perpendicular direction. Our Rayleigh quotient and work so far all hold for a material where Poisson's Ratio is zero. Most real-world materials have $\sigma\in[0,1/2]$, although there exist some materials with negative Poisson's Ratio. 

We will assume $\sigma \in (-\infty,1)$ for mathematical reasons made clear below. The generalized Rayleigh quotient is given by
\begin{equation}\label{RQsigma}
Q[u] = \frac{\int_\Omega (1-\sigma) |D^2 u|^2 + \sigma(\Delta u)^2 + \tau |D u|^2\,dx}{\int_\Omega |u|^2\,dx}
\end{equation}
and reduces to our previous quotient when $\sigma=0$. Note also that if we imposed clamped boundary conditions on our plate, Fact~\ref{Hessianwithclamped} would then allow us to write the Hessian term as $(1-\sigma) |\Delta u|^2$. Dependence on $\sigma$ would thus vanish from the Rayleigh quotient. Thus, when the plate is clamped, Poisson's ratio has no effect on the spectrum.

Free and supported plates with nonzero Poisson's ratio have been considered previously: see \cite{P58}, \cite{NF55}, and \cite{LW73}. Payne \cite{P58} considered arbitrary Poisson's ratio but zero tension. 

We can show coercivity of our new quotient $Q$ for $\sigma<1$ by discarding the Laplacian term and repeating our previous argument:
\begin{align*}
a(u,&u)+K\|u\|_{L^2}^2\\
&=(1-\sigma)\|D^2u\|^2+\sigma\|\Delta u\|^2 +\tau\|D u\|^2 +K \|u\|^2\\
&\geq(1-\sigma-\delta)\|D^2u\|_{L^2}^2+\left(\frac{\delta}{\epsilon}-\delta-|\tau|\right)\|Du\|^2_{L^2}+\left(K-\frac{C\delta}{\epsilon^2}-\delta\right)\|u\|_{L^2}^2.
\end{align*}
This argument works for each value of $\tau$ provided we choose $\delta$ appropriately and the constant $K$ large enough. Note that when $\sigma = 1$, we lose the entire Hessian term and hence coercivity; this fact is why we only consider $\sigma < 1$.

Following our earlier derivation, we obtain the same eigenvalue equation
\begin{equation*}
\Delta \Delta u - \tau \Delta u = \omega u,
\end{equation*}
along with new natural boundary conditions on $\partial\Omega$:
\begin{align*}
&Mu|_{\partial\Omega} := \left.(1-\sigma)\frac{\partial^2u}{\partial n^2}+\sigma\Delta u\right|_{\partial\Omega} = 0\\
&Vu|_{\partial\Omega} := \left.(\tau\frac{\partial u}{\partial n}-(1-\sigma)\sdiv\left(\sproj\left[(D^2u)n\right]\right)-\frac{\partial\Delta u}{\partial n}\right|_{\partial\Omega} = 0.
\end{align*}
As before, $n$ is the outward unit normal to the surface, $\sdiv$ the surface divergence, and $\sproj$ the projection onto the tangent space of $\partial\Omega$.

The generalization to nonzero $\sigma$ does not change the eigenvalue equation and hence the general form of solutions is preserved. However, the change in the Rayleigh quotient affects the proof of Theorem~\ref{thm2}, which identified the fundamental mode of the ball.  This in turn affects the proof of the isoperimetric inequality. Below I will discuss where the proof of Theorem~\ref{thm2} breaks down, first looking at the failure of the argument for the $l\geq 1$ case, and then observing that a modification of the argument in the $l=0$, $1$ case can be made for nonzero $\sigma$.

Considering eigenfunctions of the form $u=R(r)\Yl$ as in the proof of that theorem and following our previous calculations, the numerator of the Rayleigh quotient $N[u]$ becomes
\begin{align*}
(1-&\sigma)\int_0^1\left((\Rpp)^2+\frac{2k+d-1}{r^2}(\Rp)^2-\frac{6k}{r^3}R\Rp+\frac{k(k-d+4)}{r^4}R^2\right)r^{d-1}\,dr\\
&\qquad+\sigma\int_0^1\left(\Rpp(r) + \frac{(d-1)\Rp(r)}{r}-\frac{kR(r)}{r^2}\right)^2r^{d-1}\,dr\\
&\qquad+\tau\int_0^1\left((\Rp)^2+\frac{R^2}{r^2}k\right)r^{d-1}\,dr
\end{align*}
with $k=l(l+d-2)$. When $\sigma=0$ and considering positive tension we were able to rewrite the numerator of the Rayleigh quotient to show was an increasing function of $k$ for $k\geq 1$ and hence an increasing function of $l$ for $l\geq 1$. Unfortunately, when our Poisson's Ratio is nonzero, we can no longer complete the square to obtain increasing functions of $k$. There is a cross-term involving both $k$ and $\Rpp$, but $(\Rpp)^2$ involves no factor of $k$, so we cannot complete a square to transform these terms into a something that is an increasing function of $k$ for all $R$.  At present I cannot see how one could rewrite these terms to have an always-increasing function of $k$; furthermore, numerical investigations with Mathematica suggest that for some choices of $\tau$ and $\sigma$, the fundamental mode for the ball corresponds to $l=2$ or greater.

The proof that the lowest eigenvalue corresponding to the eigenfunctions $(j_1(ar)+\gamma b_1(br))Y_1$ is lower than that corresponding to $(j_0(ar)+\gamma b_0(br))Y_0$ still holds; we are able to show that $MV_1(a)$ has a root in $(0,\ainf)$ and that the first root of $MV_0(a)$ falls after $\ainf$. So while we are unable to prove the precise form of the fundamental mode, we still have established that the fundamental mode corresponds to some $l\geq 1$, and hence has angular dependence. Thus we have lost a lot of needed information about our radial function, and only know that 
\[
\rho = j_l(ar)+\gamma i_l(br)\qquad\text{on $[0,1]$}
\]
for some positive $l$ not necessarily $l=1$.  Hence our proof for the isoperimetric inequality is also adversely affected when $\sigma\neq0$, and I cannot see how to complete it. It is not even clear whether the fundamental tone should be maximal for the ball when $\sigma\neq0$.

\section*{Harmonic mean of low eigenvalues}
In two dimensions, Szeg\H o was able to prove a stronger statement of the Szeg\H o-Weinberger inequality using conformal mappings \cite{S50,Serrata}. Specifically, he proved that the sum
\[
\frac{1}{\mu_1}+\frac{1}{\mu_2}
\]
is minimal for a disk. In other words, the harmonic mean of $\mu_1$ and $\mu_2$ is maximal for the disk.  Our investigation in Chapter~\ref{fundmodtens} with the moment of inertia suggests a similar result for the free plate, since the moment of inertia is minimal for a ball. That is, for the free plate, we conjecture
\[
\frac{1}{d}\sum_{i=1}^d\frac{1}{\omega_i(\Omega)}\geq \frac{1}{\omega_1(\Ostar)}.
\]

\section*{Curved spaces}
We have taken our region $\Omega$ to be in Euclidian space $\RR^d$, but we could consider the same eigenvalue problem on a region in spaces of constant curvature: the sphere and hyperbolic space. Other eigenvalue inequalities have been proven in these spaces \cite{Asummary}. In particular, the Szeg\H o-Weinberger inequality for was proved for domains on the unit sphere by Ashbaugh and Benguria \cite{ABSW}. Another direction of generalization would be Hersch-type bounds for metrics on the whole sphere or torus; see \cite{hersch}.

%% file: calcfact.tex
\chapter{Calculus facts}\label{calcfact}
This appendix collects some calculus facts used in the proof of Theorem~\ref{thm1}. Below, $\rho(r)$ is any smooth function of the radial coordinate and the $x_i$ are the usual rectangular coordinates. We also define functions 
\[
u_k=x_k\frac{\rho(r)}{r}.
\]

\begin{fact}\label{derivs}\label{scoord} We have the sums
\begin{align*}
\sum_{k=1}^d|u_k|^2&=\rho^2\\
\sum_{k=1}^d|Du_k|^2&=\frac{d-1}{r^2}\rho^2+(\rp)^2\\
\sum_{k=1}^d|D^2u_k|^2&= (\rpp)^2+\frac{3(d-1)}{r^4}(\rho-r\rp)^2\\
\sum_{k=1}^d(\Delta u_k)^2 &=\left((n-1)\frac{A}{r^2}-\rpp\right)^2.
\end{align*}
\end{fact}
\begin{proof}
The first is immediate from $x_1^2+\dots+x_d^2=r^2$.
Now note that
\[
\frac{\partial r}{\partial x_k} = \frac{x_k}{r},
\]
so that
\[
\frac{\partial u_k}{\partial x_i} = \delta_{ik}\frac{\rho}{r}-\frac{x_ix_k}{r^3}(\rho-r\rp).
\]
Thus
\[
\frac{\partial^2 u_k}{\partial x_i\partial x_j}
=-\frac{x_k\delta_{ij}+x_i\delta_{jk}+x_j\delta_{ik}}{r^3}(\rho-r\rp)+\frac{x_ix_jx_k}{r^3}\left(\frac{3(\rho-r\rp)}{r^2}+\rpp\right),
\]
and
\[
\frac{\partial^2u_k}{\partial x_i^2}=-\frac{x_k+2x_i\delta_{ik}}{r^3}(\rho-r\rho^\prime)+\frac{x_i^2x_k}{r^3}\left(\frac{3(\rho-r\rp)}{r^2}+\rpp\right).
\]
Write $A = (\rho-r\rp)$ and $B =3r^{-2}(\rho-r\rp)+\rpp$. Then the sum of the gradients becomes
\begin{align*}
\sum_{k=1}^d|Du_k|^2 &=\sum_{k=1}^d\sum_{i=1}^d\left(\delta_{ik}\frac{\rho}{r}-\frac{x_ix_k}{r^3}A\right)^2\\
&=(d-1)\frac{\rho^2}{r^2}+(\rp)^2,
\end{align*}
by using the $\delta_{ik}$ to simplify.
The sum of the Hessian terms becomes
\begin{align*}
\sum_{k=1}^d|D^2u_k|^2&=\sum_{k=1}^d\sum_{i=1}^d\sum_{j=1}^d\left(-\frac{x_k\delta_{ij}+x_i\delta_{jk}+x_j\delta_{ik}}{r^3}A+\frac{x_ix_jx_k}{r^3}B\right)^2\\
&=(\rpp)^2+3(d-1)\frac{A^2}{r^4},
\end{align*}
after collapsing the sums and simplifying. Similarly, the sum of the Laplacians becomes
\begin{align*}
\sum_{k=1}^d(\Delta u_k)^2 &=\sum_{i,j,k=1}^d\left(-x_k\frac{1+2\delta_{ik}}{r^3}A+\frac{x_i^2x_k}{r^3}B\right)\left(-x_k\frac{1+2\delta_{k,j}}{r^3}A+\frac{x_j^2x_k}{r^3}B\right)\\
&=\left((d-1)\frac{A}{r^2}-\rpp\right)^2.\qedhere
\end{align*}
\end{proof}
The next two facts express the norm of the Hessian matrix in terms of the Laplacian and a divergence.
\begin{fact}\label{ptwise} If $u(x)$ is $C^3$-smooth, then
\begin{align*}
|D^2u|^2 &= \frac{1}{2}\Big(\Delta|Du|^2 -D(\Delta u)\cdot D\ubar-D(\Delta\ubar)\cdot Du\Big)\\
&=\frac{1}{2}D\cdot\Big(D|Du|^2-\Delta uD\ubar-\Delta\ubar Du\Big)+|\Delta u|^2.
\end{align*}
Here $u$ could be complex-valued, with $\ubar$ denoting its complex conjugate.
\end{fact}
\begin{proof} We have
\begin{align*}
D\cdot\Big(D&|Du|^2-\Delta uD\ubar-\Delta\ubar Du\Big)+2|\Delta u|^2\\
&=\Delta|Du|^2-D(\Delta u)\cdot D\ubar-\Delta u \Delta\ubar-D(\Delta\ubar)\cdot Du-\Delta\ubar\Delta u+2|\Delta u|^2\\
&=\Delta|Du|^2-D(\Delta u)\cdot D\ubar-D(\Delta\ubar)\cdot Du\\
&=\sum_{k,l=1}^d\Big(2u_{x_kx_l}\ubar_{x_kx_l}+u_{x_k}\ubar_{x_kx_lx_l}+u_{x_kx_lx_l}\ubar_{x_k}\Big)\\
&\qquad-\sum_{j,k=1}^du_{x_kx_kx_l}\ubar_{x_l}-\sum_{j,k=1}^d\ubar_{x_kx_kx_l}u_{x_l}\\
&=2|D^2u|^2.\qedhere
\end{align*}
\end{proof}
\begin{fact}\label{Hessianwithclamped} If $u\in C^3(\overline{\Omega})$ and $Du=0$ on $\partial\Omega$, then
\[
\int_\Omega|D^2u|^2\,dx=\int_\Omega|\Delta u|^2\,dx.
\]
\end{fact}
\begin{proof} Integrate Fact~\ref{ptwise} and apply the divergence theorem.
\end{proof}
Thus if $u$ satisfies clamped boundary conditions $u=\partial u/\partial n=0$, then the Hessian term in the Rayleigh quotient \eqref{RQ} can be replaced with $|\Delta u|^2$.